\documentclass[10pt,reqno]{amsart}

\usepackage{mathabx}
 \usepackage{mathrsfs}               % \mathscr
\usepackage{amssymb}
\usepackage{graphicx, amsmath}
\usepackage{amsfonts}
\usepackage{dsfont}
\usepackage{amssymb}
\usepackage{fancyhdr}
\usepackage{a4wide}
\usepackage{xcolor}
\usepackage[colorlinks]{hyperref}
\usepackage{scalerel}
\usepackage{scalerel}
\usepackage{bm}
\usepackage{braket}

\usepackage{enumerate} 
\usepackage{enumitem}
\newcommand{\ve}{\varepsilon}
\newcommand{\vp}{\varphi}
\newcommand{\bmu }{{\bm \mu }}
\newcommand{ \bX  }{  {\bm X} }
\newcommand{ \bG  }{  {\bm G} }
\newcommand{\C}{ \mathscr{C} }
\newcommand{ \bC }{  { \bm \C  } }
\renewcommand{\c}{ \mathcal{C} }
\newcommand{\bF}{{ \bm F}  }

\renewcommand{\bf}{ {\bm f} }
\newcommand{\tr}{\textnormal{Tr}}

\newcommand{\<}{\left\langle}
\renewcommand{\>}{\right\rangle}
\renewcommand{\d}{\mathrm{d}}

\renewcommand{\L}{\mathscr{L}}

\newcommand{\R}{\mathbb{R}}

\newcommand{\N}{\mathbb{N}}

\newcommand{\calE}{\mathcal{E}}

\newcommand{\calS}{ \mathcal{ S } }
\newcommand{\M}{ \mathcal M  }

\renewcommand{\S}{\mathbb{S}}
\newcommand{\1}{\mathds{1}}

\newcommand{\h}{\textbf{(H1)}}
\newcommand{\hh}{\textbf{(H2)}}
\newcommand{\hhh}{\textbf{(H3)}}

\newcommand{\calF}{ \mathcal{F} }

\newcommand{\calR}{\mathcal{R}}

\newcommand{\calI}{\mathcal{I}}
\newcommand{\bc}{ \, \textbf{,} \, }
\newcommand{\sym}{{ \textnormal{sym}  }}

\newtheorem{theorem}{Theorem}[section]
\newtheorem{lemma}{Lemma}[section]
\newtheorem{proposition}{Proposition}[section]

\newtheorem{condition}{Condition}
\newtheorem*{hypothesis}{Hypothesis}

\theoremstyle{definition}   
\newtheorem{definition}{Definition}

\theoremstyle{remark}
\newtheorem{remark}{Remark}[section]
\newtheorem{remarks}[theorem]{Remarks}

\newcounter{listi}
\newenvironment{remarklist}{\begin{list}{{\rm(\roman{listi})}}{%
			\setlength{\topsep}{0mm}\setlength{\parsep}{1mm}\setlength{\itemsep}{0mm}%
			\setlength{\labelwidth}{1.3em}\setlength{\leftmargin}{1.2em}\usecounter{listi}%
	}}{\end{list}}

\renewcommand\ldots{\makebox[1em][c]{.\hfil.\hfil.}}

\let\oldtocsection=\tocsection

\let\oldtocsubsection=\tocsubsection

\let\oldtocsubsubsection=\tocsubsubsection

\renewcommand{\tocsection}[2]{  \hspace{0em}\oldtocsection{#1}{#2}}
\renewcommand{\tocsubsection}[2]{  \hspace{2em}\oldtocsubsection{#1}{#2}}
\renewcommand{\tocsubsubsection}[2]{\hspace{2em}\oldtocsubsubsection{#1}{#2}}

\numberwithin{equation}{section}

\usepackage[margin=3cm]{geometry}
\begin{document}

	\begin{abstract}
		In this work, we generalize  M. Kac's original many-particle binary stochastic model   to derive a space homogeneous Boltzmann equation that includes a  linear combination of higher-order 
		collisional  terms. 
		First, we prove an abstract theorem
		about convergence
		from  a   finite  hierarchy  to an infinite  hierarchy
		of  
		coupled equations. 
		We apply this convergence theorem on hierarchies
		for marginals corresponding to  the generalized Kac model mentioned above.
		As a corollary, we prove propagation of chaos for the marginals associated to the generalized Kac model. 
		In particular, the first marginal converges towards the  solution of a 
		Boltzmann  equation including  interactions up to a finite order, and whose  collision kernel is of Maxwell-type with cut-off. 
	\end{abstract}

	\title[The Boltzmann Equation with Higher-order collisions]{Derivation of a Boltzmann equation with higher-order collisions from a generalized Kac model}

	\author{Esteban C\'ardenas}	
	\address[Esteban C\'ardenas]{Department of Mathematics,
		University of Texas at Austin,
		2515 Speedway,
		Austin TX, 78712, USA}
	\email{eacardenas@utexas.edu}
	
	\author{Nata\v{s}a Pavlovi\'c}

	\address[Nata\v{s}a Pavlovi\'c]{Department of Mathematics,
		University of Texas at Austin,
		2515 Speedway,
		Austin TX, 78712, USA}
	\email{natasa@math.utexas.edu}

	\author{William Warner}

	\address[William Warner]{Department of Mathematics,
		University of Texas at Austin,
		2515 Speedway,
		Austin TX, 78712, USA}
	\email{billy.warner@utexas.edu}

	\maketitle

	{\hypersetup{linkcolor=black}
		\tableofcontents}

	\section{Introduction}
	\label{section intro}
 		The aim of this paper is to  derive a  Boltzmann  equation for Maxwell molecules that incorporates higher-order collisions; 
 		we achieve that by generalizing M. Kac's original stochastic binary model \cite{Kac 1956}  via allowing multi-particle interactions. 
 With this purpose in mind, let us consider a space homogeneous gas of indistinguishable particles, moving in $d$-dimensional Euclidean space.
 The system is to be described by the probability density $f = f(t,v)$ 
 	of  finding a single particle with velocity  $v \in \R^d  $ at time $t\geq0$. 
 	The resulting Boltzmann-type equation  will be of the form 
 	\begin{equation}\label{boltzmann type equation}
 	\partial_t f 
 	=
 	\beta_1  \, 
 	Q_1 (f)
 	+
 	\beta_2     \, 
 	Q_2(f,f)
 	+ 
 	\cdots 
 	+
 	\beta_M     \, 
 	Q_M  (
 	{   f , \ldots, f } 
 	) \, .
 	\end{equation}
 	where $(\beta_K)_{K=1}^M$ is a normalized set of coefficients: $\sum  \beta_K=1$. 
 	Here, $    M   \in \N $ is the highest-order collision that will be relevant in our system, and 
 	$ (  Q_K  )_{ K     = 1 }^M$ are the $ K $-th collisional operators,  modelling the interactions between $K$ particles. 
 %	The collision kernels we derive will be independent of the velocities, and integrable with respect to the scattering angles--these are the so-called Maxwell-type interactions with cut-off. 
 	
 	\vspace{1mm}
 	
 	Since the Boltzmann equation was introduced by L. Boltzmann \cite{Boltzmann 1872} and J. Maxwell \cite{Maxwell 1867}, it has been the target of many mathematical studies. 
 	In particular, the problem of rigorously deriving a  Boltzmann equation with \textit{binary} interactions (of Maxwell-type) was first addressed by  M. Kac in his foundational work
 	\cite{Kac 1956}. 
 	By setting up an appropriate $N$-particle stochastic process, 
 	Kac was able to show that an equation of the form \eqref{boltzmann type equation} 
 	-- with the right hand side containing only the $Q_2$ term --
 	emerges from the many-particle dynamics in the $N \rightarrow \infty $ limit. 
 	The framework introduced in \cite{Kac 1956} is now known as the \textit{Kac model}, 
 	and there is an active field of research around it; 
 	its simplicity is a fertile playground for studying subtle questions that are otherwise very difficult to approach  in more complex models arising from kinetic theory. 
 	\textit{Propagation of chaos}, \textit{entropy production}, 
 	\textit{relaxation towards equilibrium} and \textit{well-posedness}
 	are among the most studied questions for the Kac model and its generalizations. 
 	For a partial survey of articles,   see e.g  
 	\cite{Kac 1956,MischlerMouhot2013,MischlerMouhotWennberg2015,BonettoLossVaidyanathan2013,Cortez2016,CortezTossounian2020,CarlenMustafaWennberg2015,Carrapatoso2015,HaurayMischler2014,Villani2003,Einav2011,BonettoGeisingerLossRied2018,DiaconisSaloff-Coste2000,Janvresse2001,CarlenCarvalhoLoss2001,CarlenCarvalhoLoss2003,Maslen2003,CarlenCarvalhoLeRouxLossVillani2010,CarlenGeronimoLoss2008,Hauray2016}
 	and references therein.

 	\vspace{1mm}

 	Derivation of the Boltzmann equation in the deterministic, space-inhomogeneous setting with hard spheres  has been  a major breakthrough in kinetic theory--the first proof in this direction was given by O. Lanford's  \cite{Lanford 1975}. 
 	More recently, this derivation program has been revisited in a modern perspective by I. Gallagher, L. Saint-Raymond  and B. Texier \cite{Gallagher Saint-Raymond Texier 2013}.  
 	On the other hand, derivations of  Boltzmann-type equations that include \textit{higher-order} collisions between the particles has just recently started to receive more attention. 
 	In \cite{AmpatzoglouPavlovic2019}, I. Ampatzoglou and the second author of this paper
 	derived the non-homogeneous Boltzmann equation for hard spheres, with  the relevant interactions being ternary. 	 In \cite{AmpatzoglouPavlovic2020}, the same authors were able to simultaneously include both binary and ternary interactions in their analysis--the   problem of including arbitrarily higher-order interactions remains open.  
 	We would also like to point out the recent work \cite{AmpatzoglouGambaPavlovicTaskovic2022} which implies that  
 	an equation of the type \eqref{boltzmann type equation} including a linear combination of collision operators can give better properties of solutions compared to the binary Boltzmann equation. Specifically, in  \cite{AmpatzoglouGambaPavlovicTaskovic2022}, I. Ampatzoglou, I. M. Gamba, M. Tasković
 	 and the second author of this paper, 
 	have shown that the simultaneous existence of binary and ternary collisions in a homogeneous Boltzmann-type equation yields better generation in time properties of moments and time decay, 
 	compared to when only binary or ternary collisions are considered. 
 	This gives additional motivation to study both derivation as well as analysis of Boltzmann equations with higher order collisions such as \eqref{boltzmann type equation}, 
 	which is what we do in this paper in the context of Kac's stochastic framework. 
 	
 	\vspace{1mm}
 	
 	More precisely,   we introduce an adaptation of M. Kac's original stochastic $N$-particle model that simultaneously includes interactions up to order $  M  \in \N, $ and prove that Eq.  \eqref{boltzmann type equation} emerges in the $N \rightarrow \infty $ limit. 
 	The model we propose is motivated by the work  of A.V Bobylev, C. Cercignani and I. M. Gamba 
 	\cite{gamba}, on well-posedness and self-similar solutions of 
 	an equation that incorporates higher-order collisions between Maxwell molecules.
 	Inspired by  \cite{Gallagher Saint-Raymond Texier 2013,AmpatzoglouPavlovic2020} 
 	we use hierarchy methods to obtain convergence from a certain finite hierarchy of equations to the infinite hierarchy, associated to the generalized Kac model. Propagation of chaos then follows as a corollary.

 	\subsection{Higher-order collisions
 	}  
 	Let us be now introduce higher-order collisions.
 	We shall not specify a concrete transformation map between pre- and post-collisional velocities, 
 	but rather work in a general setting that satisfies three conditions, given in the Hypothesis below.
 	We present some examples in Section \ref{section applications}. 
 	
 	\vspace{1mm}
 	
 	\noindent \textit{The Transformation Law.}
 	For every   $ K  = 1 , \ldots ,   M  $  
 	we assume that we are given a measurable space $\S_K$ with a probability measure $b_K$, 
 	together with a measurable map 
 	$$
 	T_K : \S_K \times \R^{d K } \rightarrow \R^{dK}  \  .
 	$$
 	We call $K$ the \textit{order} of the collision, 
 	$\S_K$ the space of \textit{scattering angles}, 
 	$b_K$ a   \textit{collision kernel} 
 	and
 	$T_K$ the \textit{transformation law}.

 	Throughout this work, we assume the following Hypothesis to be satisfied.
 	Let us denote   by $S_K$   the group of permutations of $K $ elements. 
We will abuse notation and use the same symbol to denote a permutation  $\sigma \in S_K$
and its action over the space $\R^{dK}$.
Namely, $\sigma $ stands for the function defined by 
 	$\sigma (V ) \equiv  (V_{\sigma(1)} , \ldots, V_{\sigma(K)}  ) $  for $V = (V_1, \ldots, V_K) \in \R^{dK}$. 
 	
 	\begin{hypothesis}\label{hypo1}
 		For  all $\omega \in \S_K $,  the map $ T^\omega_K \equiv T_K (\omega, \cdot ): \R^{dK } \rightarrow \R^{dK }$ is  linear.
 		Additionally:
 		\begin{enumerate}[label=\textnormal{\textbf{(H\arabic*)}}]
 			\item  $T^\omega_K$ is an isometry. 
 			\item 	For all $ \vp  \in C (\R^{dK })$   it holds that 
 			\begin{equation}\label{K-order collision}
 			\int_{\S_{K  }}  \vp  \big[ (T_K^\omega)^{-1 }  V    \big]    \d b_K(\omega )
 			=
 			\int_{\S_{K  }}  \vp [  T_K^\omega  V    ]  \d b_K(\omega ) \  , \qquad \forall  \, V \in \R^{dK } \ . 
 			\end{equation}
 			\item For all  $  \sigma  \in S_K  $ and 
 			$ \vp  \in C (\R^{dK })$
 			it holds  that   
 			\begin{equation}\label{symmetry condition}
 			\int_{\S_K}  \vp 
 			\big[
 		   (\sigma \circ T^\omega_K\circ \sigma^{-1}) \, V 
 		\big]
 			\
 			\d b_K (  \omega  ) 
 			=
 			\int_{\S_K}  \vp    \big[       T^\omega_K V      \big]
 			\
 			\d b_K (  \omega  )   \   , \qquad \forall \,  V \in \R^{d K }  \ . 
 			\end{equation}
 		\end{enumerate}
 	\end{hypothesis}

The conditions introduced above arise   when considering elastic collisions between particles whose pre- and post-collisional velocities are 
related by the formula
\begin{equation} \label{transformation law}
(v_1 , \ldots, v_K) 
\longmapsto
(v_1^* , \ldots, v_K^*) 
\equiv 
T_K^\omega (v_1, \ldots, v_K) \ , 
\end{equation}
where $\omega \in \S_K  $ is   a parameter that labels      the directions in which
the particles interchange momentum. With this interpretation in mind we can give   physical relevance to the above hypotheses. $\h$ states that there is \textit{conservation of kinetic energy}. $\hh$ states that, up to an average over the set of scattering angles, the transformation law $T_K$ is an \textit{involution}. $\hhh$ states that, up to an average over the set of scattering angles, the transformation law $T_K$ does not depend on the labeling of the particles, e.g there is no preferred order in which the particles can enter a  collision.

\begin{remarks}
A few comments are in order regarding $\h-\hhh$. 
\begin{remarklist}
	\item 
		Even though our methods can be    adapted to include transformation laws that do not satisfy
	$\hhh$, we include it to make the exposition simpler.
	Similar assumptions have previously been made in the literature;  
	see for instance Definition 2.1 \textit{(iv)}
	in \cite{CarlenCarvalhoLoss2019} for an example
	in the context of the Quantum Kac Model. 
	\item 
	 From a mathematical point of view, we include $K=1$  since it presents no additional difficulties. 
	Physically, it does not correspond to collisions between the particles, but can be understood as an interaction between a single particle and its medium;  
	a famous example is the  thermostat model  \cite{BonettoLossVaidyanathan2013}.
	\item 
		In order to accommodate certain models, 
	we \textit{do not} require  {conservation of momentum}   to hold: 
	\begin{equation}
	\textstyle 
	\sum_{i =1}^K v_i ^* = \sum_{i=1}^K v_i \ .
	\end{equation}
	We refer the reader to Section \ref{section applications} for details.
\end{remarklist}
\end{remarks}

\noindent \textit{The Collisional Operators.}
In this setting, the transformation law $T_K$ 
defines
 	the collisional operators
 	 		$Q_K: \prod_{i=1}^K L^1(\R^d )  \rightarrow L^1 (\R^d ) $
present in the Boltzmann-type equation
  \eqref{boltzmann type equation}
 	that we derive.
 	More precisely, for $K\geq 2$ these operators are of the form 
	\begin{align}
	\label{collision operators}
	Q_K 
	( f_1, \ldots , f_K )  ( v_1  )
	&    :=
	K
	\int_{\S_{K  } \times \R^{d ( K  -  1  )   }}  
	\Big(
	(   \otimes_{\ell =  1 }^K  f_\ell  )
	(T^\omega_{K  }    V    )
	-
	(   \otimes_{ \ell =  1 }^K  f_\ell  )			 
	(   V   )
	\Big)      \d b_K( \omega )      d v_{2 } \ldots  d v_{K  } 
	\end{align}
 	and analogously for $K=1$, with $\S_K\times \R^{d(K-1)}$ being replaced by $\S_1$. 
 	Notice above that the kernel of $Q_K$ is independent of the \textit{relative velocities}, and 
 	is \textit{integrable} with respect to the scattering angles. 
 	In the context of kinetic theory, such a model can be interpreted as a  gas of Maxwell molecules with an angular cut-off.

%
%
%
%\begin{remark}
%	Even though our methods can be    adapted to include transformation laws that do not satisfy
%	$\hhh$, we include it to make the exposition simpler.
%	Similar assumptions have previously been made in the literature; see for example  [..., Remark 2.2]. 
%\end{remark} 
%\begin{remark} 
%	 From a mathematical point of view, we include $K=1$  since it presents no additional difficulties. 
%	Physically, it does not correspond to collisions between the particles, but can be understood as an interaction between a single particle and its medium;  
%	a famous example is the  thermostat model  \cite{LossBonetto2013}.
%	
%\end{remark}
%
%
%\begin{remark}
%	In order to accommodate certain models, 
%	we \textit{do not} require  {conservation of momentum}   to hold: 
%	\begin{equation}
%	\textstyle 
%	\sum_{i =1}^K v_i ^* = \sum_{i=1}^K v_i \ .
%	\end{equation}
%	We refer the reader to Section \ref{section applications} for details.
%\end{remark}

\vspace{2mm}
 	
\noindent \textit{{Extension of the Transformation Law to $N$ Particles}\label{extension subsection}}.
Let $K \in \{ 1, \ldots, M  \}$ be a fixed order of collision. 
We would like to define  collisions of order $K$, that happen in a system of $N$ particles; 
their velocities will be recorded by the so-called \textit{master vector}
$V = (v_1, \ldots, v_N) \in \R^{dN}$. 
In order to \textit{select} the particles that undergo a collision, let us denote by 
 	\begin{equation}
 	\label{index set} 
 \calI(K) := \big\{    (i_1 , \ldots,  i_K) \in \{ 1, \ldots , N\}^K  \ :   i_j \neq i_\ell \text{ for } j \neq \ell \big\}
 	\end{equation}
 	the set of all pairwise different    indices  contained in   $\{ 1, \ldots , N\}^K$. 
 	Note that we do not require the indices to be ordered, i.e. $ i_1 < \cdots < i_K$ may not hold.
 
 \vspace{1mm}
 
 Next, let us fix a collection of indices $    (i_1 , \ldots, i_K )   \in \calI(K)$ and
consider  a permutation
 $\sigma \in S_N$
 	 satisfying $\sigma (1) = i_1 , \cdots , \sigma (K) = i_K$. 
 	 Then, 
 we will work extensively with the new linear map 
 	\begin{equation}
 	\label{T extension} 
 		T^\omega_{i_1 \cdots i_K}  := 
 		\sigma \circ ( T^\omega_K \times id_{\R^{d (N-K) }} ) \circ \sigma^{-1 }
 		:  \R^{dN} \rightarrow \R^{dN}  \   , \qquad 
 		\omega \in \S_K \ .
 	\end{equation}
 	In   words, the map  
 		$T^\omega_{i_1 \cdots i_K} $
 		selects the particles labeled by indices $(i_1, \ldots, i_K)$
 		and updates their 
 		velocities according to the transformation law  \eqref{transformation law}, i.e.
 		$(v_{i_1} , \ldots, v_{i_K}) \mapsto (v_{i_1}^* , \ldots, v_{i_K}^*)$, while leaving the rest invariant. 

 \begin{remark}
 For the special case in which the indices are ordered, meaning that 
$i_1 < \cdots  < i_K$, 
one can   write 
for   $V = (v_1, \ldots ,v_N) \in \R^{dN }$: 
\begin{equation}
    T_{i_1 \cdots i_K}^\omega
    V
    =
    (
    v_1 , \cdots ,v_{i_1 -1 },  v_{i_1}^* , v_{i_{1}+1},
    \cdots, 
    v_{i_K -1 } , 
    v_{i_K}^*,
    v_{i_K +1 }
    , 
    \cdots, 
    v_n
    ) \ , 
\end{equation}
  where $ (v_{i_1}^* , \ldots, v_{i_K}^* ) \equiv T_K^\omega (v_{i_1} , \ldots, v_{i_K}) \in \R^{dK }. $
 	   
 \end{remark}

  \subsection{Generalized Kac Model}
 As in M. Kac's original approach for deriving a binary Boltzmann equation, we shall construct a Markov process describing
 the $N$-particle system
 and study the relevant \textit{Master Equation}
 governing its dynamics. 
 	Details of this  construction  can be found in Section \ref{master equation section}.
 
\vspace{1mm}

  Our Master equation is then given by
 	 	\begin{equation}
 	 	\label{master eq}
 	\begin{cases}
 	\partial_t f_N = \Omega f_N \\
 	f_N( 0   ) = f_{N,0} \in 
 	L^1_{\mathrm{sym}}(   \R^{dN} ) 
 	\end{cases} \,  ,
 	\end{equation}
 	where $L^1_{    \mathrm{sym}  }$ stands for the space of $L^1$ functions, invariant under permutation of their variables.
 	The generator $\Omega: L^1 (\R^{dN}) \rightarrow L^1 (\R^{dN})$ is the bounded linear operator determined  by the formula 
 	\begin{equation}\label{omega}
 	\Omega \,  f 
 	= N \!  
 	\sum_{K=1 }^M
 	\beta_K 
 	\! \!  \! \! 
 	\sum_{ (i_1 , \cdots   , i_K ) \in \calI(K) } 
\!  
 	\frac{1}{  K !  \binom{N}{ K  }   } 
 	\int_{\S_K   }
 	\big( 
 	f \circ  T^\omega_{  i_1  \cdots  i_K   }  - f  \big)  \, \d  b_K (\omega )    \ , 
 	\qquad 
 	f \in L^1(   \R^{dN} ) \ ; 
 	\end{equation}
we recall that the normalized coefficients $(\beta_K)_{K=1}^M$ where first introduced in \eqref{boltzmann type equation}. 
We note that,  since $\Omega$ is a bounded linear operator, 
the solution of the Master equation has regularity   $ f_N \in C^\infty_t (L^1_V) $.  
 	
 	% 
 	% \begin{remark} 
 	%The solution of Eq. \eqref{master equation}  may  be expressed as a norm convergent power series. Namely, it holds that 
 	%$
 	%f_N (t) = \sum_{n = 0 }^\infty   (  t^n / n! )  \,   \Omega^n \, f_{N,0} . 
 	%$
 	%As a historic remark, let us  mention that it was this representation  that M. Kac used in his original work \cite{Kac 1956} to show that   \textit{propagation of chaos} holds. 
 	%We will not make use of this formula, but instead work with the BBGKY hierarchy associated to the Master equation \eqref{master equation} and follow the analysis of \cite{Gallagher Sain-Raymond Texier 2013}. 
 	% \end{remark} 

 	\subsection{Our results in a nutshell}
 	Let us briefly  explain the three main results presented in this paper. 
 	
 \vspace{2mm}
 	
    \textit{(1) Abstract Convergence of Hierarchies}. 	
Let  $  (  X^{(s)}  )_{s \in \N}$ be a collection of Banach spaces. 
For each $N \in \N $
 we  consider the      following    system of equations 
 	\begin{equation} 
 	\label{finite hierarchy} 
 	\partial_t  f_N^{(s)} 
 	\ = \ 
 	\C_{s , s}^N f_N^{ ( s ) } 
 	\  + \ 
 	\cdots 
 	\ +  \ 
 	\C_{s, s+ M-1  }^N 
 	f_N^{ ( s + M-1 )  }  \ , 
 	  \quad s \in \N
 	\end{equation}
 	which we call    the   \textit{N-hierarchy}. 
 	Here, each unknown 
 	$f_N^{(s)}  :  [0,T] \rightarrow X^{(s)} $
is a  time-dependent,  vector-valued quantity,
 	and each   operator 
 	$\C_{s,s+k}^N : X^{(s+k)} \rightarrow X^{(s)}$
 	is linear, bounded,  and its  
 	operator norm 
 	grows at most linearly with $s \in \N$,   uniformly in $ N \in \N$.  
 	We show that if--in an appropriate sense and under additional mild assumptions--the following limits hold
$$
 	  \lim_{ N  \rightarrow \infty }  \C_{s ,s +k}^N = \C_{s ,s+ k }^ \infty    
 	  \quad 
 	  \text{ and }
 	  \quad 
 		   \lim_{ N  \rightarrow \infty } f_N^{(s)} ( 0) = f^{(s)} (0)  
 		   \qquad \forall s\in \N  , \ \forall 0 \leq k \leq M -1  \  ,
$$
 	then it also holds   for later times $t \in [0,T]$ that 
 	$$ 
 	\lim_{N \rightarrow \infty } f_N^{ (s)} (t) = f^{(s)} (t),
 	$$
 	where 	the limiting objects satisfy the  associated  infinite system of equations
 	\begin{equation}
 	\label{infinite hierarchy}
 	\partial_t f^{(s)}		\nonumber 
 	\  = \ 
 	\C_{s , s }^\infty  f^{( s) } 
 	\ + \  \cdots 
 	\ + \ 
 	\C_{s, s+ M-1  }^\infty  f^{ ( s + M-1 ) } 
 	\ , 
 	\qquad s \in \N 
 	\end{equation}
 	which we call the \textit{infinite hierarchy}.   
 	See Definition \ref{definition 3}, 
 	Definition \ref{definition convergence} and
 	Theorem \ref{thm convergence} for details. 
 	
 	\vspace{2mm}
 	
 	  \textit{(2) BBGKY to Boltzmann Hierarchy}. 
 	Starting from the solution $f_N  $  of the Master equation Eq. \eqref{master eq}, 
 	we show that its sequence of marginals $  f_N^{(s)}  $  
 	(defined through a partial trace procedure, see \eqref{def partial trace}) 
 	satisfies a finite system of equations of the form Eq. \eqref{finite hierarchy}, which we shall refer to as the \textit{BBGKY hierarchy}. 
 	Under our assumptions on the transformation law $T^\omega_K$ and  the kernel $\d b_K(\omega)$, 
 	every condition of the abstract convergence result is satisfied. 
 	Consequently, we can prove that there is convergence to an infinite hierarchy, which we shall refer to
 	as the \textit{Boltzmann hierarchy}. A precise statement can be found in Theorem \ref{thm bbgky to boltzmann}.

 		\vspace{2mm}
 	
 	   \textit{(3) Propagation of Chaos.}
 This result concerns the {derivation} of the Boltzmann equation \eqref{boltzmann type equation}.
Namely, we assume that 
the  initial data of  the Master equation
$f_{N,0} \in L^1_{\mathrm{sym}}(  \R^{dN })$	
is  such that its sequence of marginals 
		$ (   f_{N,0}^{ (s) }  )_{s  \in \N }$ 
		converges weakly
		to a  tensor product
		$ (  f_0^{\otimes s }   )_{s \in \N } $ for some $f_0 \in L^1(\R^d )$. 
		We then prove that	for all    $t\geq0 $, it holds in the weak sense that
		\begin{equation}  
		\label{limit of marginals}
		\lim_{N \rightarrow \infty }
		\    f_N^{(s)} (t, \cdot )    
		=
	  f(t , \cdot )^{\otimes s}   
		\end{equation}
		where $f(t,v )$ is the solution of the Boltzmann equation \eqref{Boltzmann equation}, with initial data $f_0 $. 
 	This is the content of Theorem \ref{thm prop of chaos}.

	\vspace{2mm}
	
	Now we provide some context for our results with respect to applications and previous works. 
	\begin{enumerate} 
\item {\it Why is our transformation law abstract? }
We  decided to require that the transformation law $T_K$ satisfies the general hypothesis $\h-\hhh$. 
This allows us to give several examples of transformation laws satisfying the hypotheses (see Section \ref{section applications}). 
Consequently, Theorem \ref{thm bbgky to boltzmann} and Theorem \ref{thm prop of chaos} apply in each of those cases. 
This is in contrast with respect to most of previous works in the field, see e.g. \cite{AmpatzoglouPavlovic2020,Gallagher Saint-Raymond Texier 2013,BonettoLossVaidyanathan2013} that typically address specific examples. 

\item {\it Why do we have Theorem \ref{thm convergence}?}
We also remark that our derivation of the generalized Boltzmann equation \eqref{boltzmann type equation}, as formulated in Theorem \ref{thm bbgky to boltzmann}, is a consequence of the abstract convergence result stated in Theorem \ref{thm convergence}. The motivation for such a level of generality is two-fold: it allows us to identify the estimates that are sufficient for the convergence process, and it provides efficient and robust notation that can be welcome when treating convergence of hierarchies.     

\item {\it Why do we have Theorem \ref{thm bbgky to boltzmann}, rather than just Theorem \ref{thm prop of chaos}? }
In contrast to Kac's original approach \cite{Kac 1956}, we prove propagation of chaos as a consequence of   convergence of hierarchies (Theorem \ref{thm convergence}). We are therefore able to handle  more general initial data-- we do not require tensorized initial data. 
	
\end{enumerate} 	

By completion of this work, the authors became aware of the works by Ueno \cite{Ueno1969} and Tanaka \cite{Tanaka1969.1,Tanaka1969.2}, 
 	in the late sixties. These works, as is the case with works on Kac's model, proved propagation of chaos for Markov processes driven by bounded generators, which include some of our models. Their proofs are based on  expansion methods, pioneered by Kac \cite{Kac 1956} and then further developed by McKean \cite{McKean1967}. As mentioned above, our result is more general in the sense that we prove convergence of hierarchies (Theorem \ref{thm bbgky to boltzmann}) for more general initial data and obtain propagation of chaos as a corollary.  Hence the approaches can be understood as complementing each other in terms of methods as well as the results.

	 \vspace{1mm}

Finally, let us mention that, in the last few decades, the problem of deriving an
 explicit 
 \textit{convergence rate}
 for the limits 
 \eqref{limit of marginals} 
 has received special attention and many models have been investigated; 
 see e.g \cite{MischlerMouhot2013,MischlerMouhotWennberg2015,GrahamMeleard1997,Cortez2016,CortezTossounian2020,Carrapatoso2015}.
 Such rate -- besides naturally depending on time $ t \in \R$, 
the number of particles $N \in \N$,
 and the order of the marginals $ s \in \N$ -- comes at the cost of requiring stronger assumptions on both 
 the initial data of the system, and the test functions. 
 Even though our methods do not provide a convergence rate, 
 we ask for minimal assumptions on these objects. 
 We believe that, for the model at hand and under appropiate assumptions, a convergence rate can be derived in the same spirit as the works mentioned above, possibly depending non-trivially on the total number of interactions $M$. 
 Such study  is beyond the scope of this paper.

	\vspace{2mm} 
	\noindent 
	\textit{Organization of the paper}. 
	In Section \ref{section main results} we give precise statements of our three main results. 
	In Section \ref{section applications} a collection of examples that fit our framework are given, and a few adaptations are mentioned. 
	In Section \ref{master equation section} we give the details of the construction of the Markov process that gives rise to the Master Equation \eqref{master eq}. 
Theorem \ref{thm convergence} is proven in  Section \ref{section proof convergence}, 
whereas Theorems \ref{thm bbgky to boltzmann} and \ref{thm prop of chaos} 
are proven in Section \ref{section proof derivation}.
Required well-posedness results are proven in Section \ref{appendix wp}, and we include  Appendix \ref{appendix markov}
for a review of the theory of Markov Processes.

\bigskip
\noindent {\textbf{Acknowledgments.}}
N.P. would like to thank Sergio Simonella and Mario Pulvirenti for discussions that inspired this research direction. Also the authors would like to thank 
Joseph K. Miller and Chiara Saffirio for helpful discussions. E.C. gratefully acknowledges support from  the Provost’s Graduate Excellence Fellowship at The University of Texas at Austin and from the NSF grant DMS-2009549. N.P. gratefully acknowledges support from the NSF under grants No. DMS-1840314, DMS-2009549 and DMS-2052789. W.W. gratefully acknowledges support from the NSF grant DMS-1840314.

	\section{Main Results}\label{section main results}
	In this section, we state in detail our three main theorems. 
	The first one is an abstract convergence result; in order to state it, 
introducing	
we dedicate the next subsection to in the necessary functional analytic spaces. 
	This point of view has the advantage of using only minimal estimates satisfied by collision operators. 
	Convergence   then happens naturally under the right assumptions.

	\subsection{The functional framework} \label{subsec-functional}
 Let  $ \{ X^{(s)}  \}_{s \in \N}$ be a collection of Banach spaces, 	and consider  their  direct sum 
	\begin{equation}
	\label{direct sum}
	X  
	:= 
\bigoplus_{s \in \N }
 X^{(s)}   \ . 
	\end{equation}
For any given $\mu \in \R  $ we consider the subspace of exponentially weighted sequences
	\begin{equation}
	\label{mu norm}
	\textstyle 
	X_{  \mu }  :=
	\big\{  F= (f^{(s)})_{s \in \N}  \in X 
	\  :   \ 
	\| F \|_{ \mu }  
	:= 
	\sup_{s \in \N} e^{ \mu s} \|  f^{(s)} \|_{   X^{(s)}} 
	< \infty  
	\big\}
	\end{equation}
We introduce time dependence as follows. 
Fix $T>0$ and consider the 
weight function  
	\begin{equation}
	 \bmu : [0,T] \rightarrow \R   \, , \  \qquad \bmu (t) := - t / T    \, . 
	\end{equation}
	We  define the Banach space of uniformly bounded, 
	time-dependent sequences  as 
	\begin{equation}
	\textstyle 
	\bX_{   \bmu } 
	:= 
	\{ 
	\bF
	: [0,T ] 
 \rightarrow X 
 \ :   \ 
 \|  \bF    \|_{\bmu} := 
 \sup_{t \in [0,T]} \|  \bF (t)     \|_{ \bmu(t) } < \infty
 \}  \ . 
	%\big\{    \bF  = (  \bf^{(s)})_{s \in \N }  : [0,T] \rightarrow X  
	%\  :  \ 
	%\| \bF   \|_{ \bmu} 
	%:=
	%\sup_{t \in [0,T]}
	%\|    \bF(t)		\|_{\bmu (t)}
	%< \infty    \big\}
	%
	%\!   \!  \quad 
	%\textnormal{ where }
	%\!  \!  \quad
	%\|   \bF   \|_{\bmu }
	%: =
	%\sup_{t \in   [0,T]   } \sup_{ s \in \N } e^{\bmu(t) s }  \|  \bf^{(s)}  (t) \|_{ X^{ (s) }  } \ . 
	\end{equation}

	\subsubsection{The N-hierarchy}

	In this subsection, 
we introduce an abstract version of the BBGKY hierarchy 
usually found in models that arise from kinetic theory. 
	Before we describe this  in detail, let us introduce the following convenient notation:
	given  $K  = 1  , \ldots, M $, 
	we work interchangeably with the \textit{lower case} quantities $m$ and $k$ defined through   
	\begin{equation}
	M = m + 1 
	\qquad
	\textnormal{and}
	\qquad 
	K = k+ 1 \ . 
	\end{equation}

We assume that for    $ N \geq M  $ and  $s   \in \N $ we are given a collection of bounded linear transformations 
	\begin{align*}
	\C_{s , s    }^N  
	& : 
	X^{(s)} 
	\longrightarrow 
	X^{(s) } \\
	& \quad  \vdots \\ 
	\C_{s , s+ m }^N   
	&  : 
	X^{(s+m)}
	\longrightarrow  
	X^{(s)}
	\end{align*}
	which we   refer to
	as the \textit{N-hierarchy} operators. 
	Intuitively, one may think of these operators    as the components of  an infinite  matrix, whose entries are non-zero only within a distance $m$ above the diagonal. 
Let us note that, in this framework,  the collection of operators $\{ \c_{s,s+k}^N \}_{s=1,k=0}^{\infty,m}$ may be infinite. In applications, however, they are usually a finite collection -- see for example Remark \ref{finite remark}.

	\vspace{1mm}

	To the collection of operators mentioned above, we associate the following    system of equations, from now on referred to as the \textit{N-hierarchy}
	\begin{align}\label{BBGKY equation}
	\begin{cases}
	\partial_t f_N^{  (s) } \,  \ = 
	\,  \  
	\, 
	\C_{s,s }^N f_N^{(s )} 
	\,  + \, 
	\cdots  \,  +  \,  
	\, 
	\C_{s, s+ m }^N f_N^{(s + m )}   &    \\
	f_N^{(s)}(0)   \!  \ = \   f_{N,0}^{ (s) } \in X^{(s)}& 
	\end{cases}   \quad  s \in \N   \ .
	\end{align}
	As will become clear during the 
	proof of Theorem \ref{thm convergence}, it will be convenient  to write the $N$-hierarchy in mild form. To this end, we consider
	the linear operator 
	$  \bC^N    :  X  \rightarrow X  $    
	defined for $F  = (f^{(s)}  )_{ s \in \N} $ as   
	\begin{align} 
	\label{C N}
	& (\bC^N F   )^{(s)} \ := \ 
	\C_{s,s  }^N 
	f^{(s )}
	\  + \ 
	\cdots  \  +  \  
	\C_{s, s+ m }^N f^{(s + m )}  \   .  
	\end{align} 
	
	\begin{definition}
	We say that $\bF_N = (f_N^{(s)})_{s \in \mathbb{N}} \in  \bX_\bmu $ is a mild solution to the \textit{N-}hierarchy
	\eqref{BBGKY equation}
	with initial condition 
	$F_{N,0} = (f_{N,0}^{(s)})_{s \in \mathbb{N}} \in X_0 $ if  
	    \begin{align}
	\label{bbgky mild solution}
	&    \bF_N(t) = F_{N,0}
	+
	\int_0^t 	
	\bC^N   
	\bF_N   
	(\tau )  \, \d \tau     \ , 
	\qquad \forall  t \in [0,T] \ . 
	\end{align} 
	\end{definition}
	
	%The system of equations \eqref{BBGKY equation} 
	%is now equivalent to the following single integral equation 
	%\begin{align}
	%\label{bbgky mild solution}
	%&    \bF_N(t) = F_{N,0}
	%+
	%\int_0^t 	
	%\bC^N   
	%\bF_N   
	%(\tau )  \, \d \tau   
	%\end{align}
	%where $ t \in [0,T]$, and where we have denoted 
	%$F_{N,0} = (f_{N,0}^{(s)})_{ s\in \N}.$

We will work under the following assumption.

		\begin{condition}  
			\label{condition 1}
			   There exist
			constants $\{ R_k\}_{k=0}^m$, independent of $s$ and $N$, such that for all
			$ k = 0 , \ldots, m $ there holds 
			\begin{equation}\label{BBGKY estimate}
			\|  \C^N_{s, s + k }  f^{(s + k  )} \|_{ X^{(s)}  } 
			\ \leq \ 
			R_k   \,  s \,  \|  f^{ (s+ k   )  } \|_{  X^{(s+k )} } 
			\, , 
			\qquad 
			f^{ (s +    k )}
			\in 
			X^{(s+ k )} \ .
			\end{equation}
	\end{condition}
 
	Under Condition \ref{condition 1}, the following well-posedness result holds. A proof can be found in Section \ref{appendix wp}.
	
	\begin{proposition}\label{corollary 1}
		Assume that the $N$-hierarchy operators  satisfy  Condition \ref{condition 1}. 
		Then, for all 
		$T
		< 
		(   \sum R_k  e^k  )^{ -1 }
		$
		and $F_{N,0} \in X_{0}$, 
		there is a unique mild solution 
		$ \bF_N \in  \bX_{\bmu }$
		to the $N$-hierarchy 
		\eqref{BBGKY equation}.		
		In addition, it holds that 
		\begin{equation}
		\| \bF_N \|_{ \bmu} 
		\leq 
		( 1 - \theta_1)^{ -1 }
		 \| F_{N,0} \|_{  0} \ , 
		 \qquad
		 \textstyle{with}
		 \qquad
		 \theta_1 = T \sum_{k=0}^m R_k e^k \in (0,1) \, .
		\end{equation}
	\end{proposition} 
	
	\subsubsection{The infinite hierarchy}

	If the   $N$-hierarchy operators 
	  admit  a formal limit when $ N \rightarrow \infty $, 
	we would like to understand the solutions of the infinite hierarchy they generate. 
	To this end, for each $s \in \N$ we will  consider a collection  of bounded linear transformations
	\begin{align*}
	\C_{s , s  }^\infty  &   :  X^{(s)}  \longrightarrow   X^{(s)}  \\
	& \vdots \\ 
	\C_{s , s+ m }^\infty    &  :  
	X^{(s+m)} \longrightarrow  
	X^{(s)} 
	\end{align*}
	which we call the \textit{infinite hierarchy operators}.
	%Similarly as before, we assume the following condition to hold. 

	\vspace{1mm}

	To these operators we associate the  \textit{infinite hierarchy},
	defined as the infinite system of equations  given by 
	\begin{align}
	\label{infinite hierarchy equation}
	\begin{cases}
	\partial_t f^{  (s) } \,  \ = \,  \ 
	\C_{s,s  }^\infty f^{( s)} 
	\,  + \, 
	\cdots  \,  +  \,  
	\C_{s, s+ m }^\infty f^{(s + m )}   &    \\
	f^{(s)}(0)   \!  \ = \   f_{0}^{ (s) } \in   X^{(s)}& 
	\end{cases} \quad s  \in \N \ . 
	\end{align}
	The  mild form of the infinite hierarchy is defined analogously. 
	Namely, we consider  the linear operator
	$  \bC^\infty     :  X  \rightarrow X $   
	defined for $F  = (f^{(s)}  )_{ s \in \N} $ as   
	\begin{align} 
	\label{C infinity}
	& (\bC^\infty F   )^{(s)} \ := \ 
	\C_{s,s }^\infty
	f^{(s )}
	\  + \ 
	\cdots  \  +  \  
	\C_{s, s+ m }^\infty  f^{(s + m )}.  
	\end{align} 

	\begin{definition}
	    We say that 
	    $\bF = (f^{(s)})_{s \in \mathbb{N}} \in \bX_\bmu $ is a mild  solution to the infinite hierarchy
	    \eqref{infinite hierarchy equation}
	    with initial condition $F_0 = (f^{(s)}_0)_{s \in \mathbb{N}} \in X_0 $ if  
	    \begin{align}
	\label{boltzmann mild solution}
	\bF(t) = F_{0}
	+
	\int_0^t 	
	\bC^\infty   
	\bF  
	(\tau )  \, \d \tau   \ , 
	\qquad t \in [0,T] \ . 
	\end{align}
	\end{definition}

	 We shall assume that the infinite hierarchy operators satisfy an analogous estimate as the one introduced in Condition \ref{condition 1}. 
	 Namely, 
	
	\begin{condition}\label{condition 2}
		There exist constants $  \{   \rho_k  \}_{k=0}^m$
		such that for all $ k   =  0 , \ldots \, , m $ and for all $s \in \mathbb{N}$  
		\begin{equation}\label{Boltzmann hierarchy estimate}
		\|  \C^\infty_{s, s + k }  f^{(s + k  )} \|_{X^{(s)}} 
		\ \leq \ 
		\rho_k 
		\,  s \,  \|  f^{ (s+ k   )  } \|_{X^{(s+k)} } \, , \qquad  f^{ (s +  k )} \in  X^{(s+k)}  \ .
		\end{equation}
	\end{condition}
	
The following well-posedness result is then available. 
A proof can be found in Section \ref{appendix wp}.
	
	\begin{proposition}\label{corollary 2}
		Assume that the infinite hierarchy operator $\C^\infty$ given in  \eqref{C infinity}
		satisfies Condition \ref{condition 2}.
		Then, 
		for  all $T < (        \sum \rho_k e^k  )^{ -1 }$ and $F_0 \in X_{     0}  $,
		there is a unique mild solution $ \bF \in \bX_{   \bmu } $
		to the infinite hierarchy  
		\eqref{infinite hierarchy equation}.
		In addition, it holds that 
		\begin{equation}
		\| \bF \|_{\bmu} 
		\leq 
			( 1 - \theta_2)^{ -1 }
		\|  F_0 \|_{  0} \, , 
		 \qquad
		 \textstyle{with}
		 \qquad
		 \theta_2 = T \sum_{k=0}^m \rho_k e^k \in (0,1)  \, . 
		 \end{equation}
	\end{proposition}

	\begin{remark}
		For the rest of the article, 
		we ask the time interval $[0,T]$
		to satisfy the following condition
		\begin{equation}\label{time}
		\textstyle 
		T 
		< 
		 m^{-1}
		 T_*
		 \quad
		  \textnormal{ with }
		  \quad 
		  T_*
		:=
		\min
		\big\{ 
		\big(   \sum_{k=0}^m R_k  e^k  \big)^{ -1 }, 
		\big(    \sum_{k=0}^m \rho_k  e^k  \big)^{ -1 }
		\big\} \  . 
		\end{equation}
	Here, $T_*$ stands for the 
	maximal time for which we can prove simultaneous well-posedness 
	of the two hierarchies; see 
	Proposition \ref{corollary 1} and \ref{corollary 2}.
	In particular,  $T_*$ is independent of the initial conditions. 
		Consequently, an iteration procedure for proving convergence for all $t \in \R $ is possible, provided global a priori bounds  are satisfied by the solutions of the finite and infinite hierarchies, respectively. 
		For the Kac model, these  bounds follow from the fact that the solution of the Master equation is the density of a probability measure. 
		The extra factor $1/m$ will be used to assure that certain integral remainder terms converge to zero. See Section \ref{section proof convergence} for details. 
	\end{remark}

	\subsection{Convergence of hierarchies}
	The notion of convergence that we are going to study is known in the literature as \textit{convergence of observables}. 
	Before we describe it, we introduce some notation.
	The bracket $  \langle  \cdot , \cdot  \rangle $ 
	stands for the  pairing between $X^{(s)}$ and its dual   
	$X^{(s)*} \equiv (X^{(s)})^* $. 
	
	\begin{definition}
	\label{definition 3}
		We introduce the two following notions of convergence. 
		\begin{enumerate}
			\item 
			The sequence 
			$(F_N)_{N =1}^\infty \in X$  
			\textit{converges pointwise weakly}
			to $F \in X$,
			abbreviated 
			$
			F_N 
			\xrightarrow{pw}
			F, 
			$
			if 
			\begin{equation}
			\lim_{N \rightarrow \infty}
			\langle 
			f_{N}^{ (s)}    \bc   \vp      
			\rangle 
			=  
			\langle 
			f^{(s)} 
			\bc  \vp    
			\rangle 
			\qquad 
			\forall s \in \N ,   \ \forall   \vp \in  	X^{(s)*} 
			\  
			\end{equation}
		where 
		$F_N = (  f_{N}^{ (s)} )_{s \in \N }$
		and 
		$F = (f^{ (s)})_{s \in \N }$ . 
			\item 
		The sequence    
			$\bF_N  
			 % = (\bf_N^{(s)} )_{s \in \N }  
			: [0,T] \rightarrow X 
			$ 
			\textit{converges in observables} 
			to 
			$\bF  
			%= ( \bf^{(s)})_{s \in \N }
			:
			[0,T] \rightarrow X  
			$ 
			if 
			for any $s \in \N $
			and any 
			$\vp \in X^{(s)^*}$
				\begin{equation}
		\lim_{N \rightarrow \infty }
		\<
	\bf_N^{(s)} (t)    ,
		\vp
		\>    
		=
		\<
		\bf^{(s)} (t)   ,
		\vp
		\> \ , 
		\end{equation}
		uniformly in $ t \in [0,T]$, 
			where 
		$\bF_N = (  \bf_{N}^{ (s)} )_{s \in \N }$
		and 
		$\bF = (\bf^{ (s)})_{s \in \N }$ . 
		
		\end{enumerate}
	\end{definition}
	
	\vspace{1mm}
	
	Let us now make precise the notion in which we understand convergence from the $N$-hierarchy operators $\C^N$ to the infinite hierarchy operators $\C^\infty$.

	\vspace{1mm}

	\begin{definition}\label{definition convergence}
		Let $X$ be the space introduced in \eqref{direct sum}. 
		We say that a sequence of operators
		$  T^N : X \rightarrow X$ \textit{converges} 
		to $  T : X \rightarrow X $ 
		if for any sequence $ F_{N} \in X $ 
		such that
		$F_N \xrightarrow{pw} F $,
		it holds true that
		$  T^N   F_{N}  \xrightarrow{pw } T F $. 
	\end{definition}
	
	\vspace{1mm}

	The following result is our first main theorem; it gives conditions under which convergence in observable occurs, from the finite to the infinite hierarchy. 
	
	\vspace{1mm}
	
	\begin{theorem}[Convergence of Hierarchies]
		\label{thm convergence}
		Assume that the $N$-hierarchy operators $\C^N$ 
		satisfy Condition \ref{condition 1},  
		and that the infinite hierarchy operators $\C^\infty$
		satisfy Condition \ref{condition 2}. 
		Let $\bF_N \in \bX_{\bmu }$ be a mild solution, corresponding to initial data $F_{N,0} \in X_0$, of the $N$-hierarchy \eqref{bbgky mild solution}, and  
		let $\bF \in   \bX_{   \bmu }$ be a mild solution, corresponding to initial data $F_0 \in X_0$, of the infinite hierarchy \eqref{boltzmann mild solution}. 
		In addition, assume that 
		\begin{enumerate}[label=\textnormal{(A\arabic*)}]
			\item $F_{N,0} \xrightarrow{pw} F_0$,
			\item $\sup_{ N \geq 1 }  \| F_{N,0} \|_{0} < \infty $
			\item $\bC^N$ converges to $\bC^\infty$  in the sense of Definition \ref{definition convergence}.
		\end{enumerate}
		Then, $\bF_N$ converges in observables to $\bF$.
	\end{theorem}

     % In the next subsection we apply this result to the generalized Kac model introduced in the Introduction. 

 %In the case which   the  Banach spaces are given by 
%$X^{(s)} = L^1(  \R^{ds} ),$ the BBGKY operators %have the form
%\begin{equation}
%\C_{s,s+k}^N =  \C_{s,s+k}^\infty + \calR_{s,s+k}^N  \ , \qquad  s \leq N \,  , \ k = 0 , \ldots,  m
%\end{equation}
%where $\calR $ is a \textit{remainder term}, whose operator norm vanishes as $N \rightarrow \infty$
%and where the    Boltzmann operators 
%have the form 
%\begin{equation}
%( \C_{s , s+ k }^\infty f^{(s+k)} )  (  V_s )
%= 
%K \beta_K 
%\sum_{ i =1 }^s 
%\int_{ \S_K \times \R^{d k }   }
%\Big(
%f^{(s+k) }
%(V_{s+k}^* )
%-
%f^{(s+k)} (V_{s+k})
%\Big)  \d b_K(\omega )
%\d v_{s +1}
%\ldots 
%\d v_{s+k}   \,  , 
%\quad 
%V_s \in\R^{ds}		\nonumber 
%\end{equation}
%where 
%$V_{s+k} = (V_s ; v_{s+1} , \ldots, v_{s+k})$ 
%and  
%$V_{s+k}^* =  T_{i \, s+1 \cdots s+k }^\omega  %V_{s+k}$ are the collection of  pre- and post-collisional velocities, respectively.

	\subsection{BBGKY and Boltzmann hierarchies}
	The next result of this paper 
	concerns the  application of  Theorem \ref{thm convergence} to our generalization of  the Kac model. 
	In order to state it,  let us first introduce the marginals of the solution of the Master equation \eqref{master eq}. 
	Indeed,  we consider      the following trace map
	\begin{equation}
	\label{def partial trace}
	\textnormal{Tr}_{s+1, \ldots, N}
	:
	L^1_{ \textnormal{sym} }	
	(\R^{dN } ) 
	\rightarrow 
	L^1_{ \textnormal{sym} }
	(  \R^{ds}) \ , \qquad s   \in \N 
	\end{equation}
where we recall that $L^1_{\textnormal{sym}}$
stands for $L^1$ 
functions invariant under permutation of their variables.
The trace map is then 
	defined for $f \in L_{\textnormal{sym}}^1 (\R^d)$
as
\begin{equation}	\label{marginals}
	\textnormal{Tr}_{s+1, \ldots, N}
	[f] (V_s)
	=
	\begin{cases}
	\int_{ \R^{d ( N -s )} }  f ( V_s , v_{s+1}, \ldots,  v_N ) \, \d v_{s + 1 }\cdots \d v_{N} \,  &  s < N  \\
	f ( V )& s =  N  \\ 
	0 & s > N 
	\end{cases} \, ,   
	\qquad 
	\   V_s \in  \R^{d s} \  .
	\end{equation}
	In particular, note that the trace map preserves permutational symmetry.

\vspace{1mm}

Let   $f_N$ be the solution of the Master equation  \eqref{master eq}. 
We now introduce its  marginals  as the   sequence  of  functions 
	 \begin{equation}
	        \label{marginals}
	        f_N^{(s) } := 
	\textnormal{Tr}_{s+1, \ldots, N} [f_N] \ , 
	\qquad s \in \N \ . 
	 \end{equation}
One may   show that the dynamics of the sequence of
	$s$-th marginals fits the abstract 
	functional framework introduced above. 
	Namely, by  letting 
	$X^{(s)} 
	=
	L^1_{\textnormal{sym}} (\R^{ds}) $
	we   will   show   in Section \ref{section proof derivation} that 
	$	f_N^{(s)}$ satisfies 
	\begin{equation}
	\label{bbgky hierarchy}
	\partial_t 
	f_N^{(s)}
	= 
	\c_{s,s}^N f_N^{(s)} 
	+
	\cdots
	+
	\c_{s,s+  m}^N f_N^{(s + m )} \ , \qquad \forall s \in \N 
	\end{equation}
	where       
	$\c_{s,s+k}^N : 
	L^1_{\textnormal{sym}}
	(\R^{d(s+k)})
	\rightarrow 
	L^1_{\textnormal{sym}} (\R^{ds})$
	are operators  that
	can be computed explicitly.  We shall refer to 
  \eqref{bbgky hierarchy} as the 
	\textit{BBGKY hierarchy}.

	\vspace{1mm}
	
	In order to display the structure of the operators $\{ \c_{s,s+k}^N \}_{s=1,k=0}^{\infty , m}$ let us first introduce some notation that will be used for the rest of the article. 
	
	\vspace{2mm}
	\noindent 
	\textbf{Notation.}
	Let $s \in \N$ and $k \in \{0 ,\ldots, m \}$.
	Given 
	$V_s  = (v_1 , \ldots, v_s) \in \R^{ds},$ 
		$v_{s+1} , \ldots, v_{s+k} \in \R^{d}, $
	 an index $ i \in \{1 , \ldots, s  \}, $	and a scattering angle $\omega \in \S_K $,
	 we  record  the   \textit{pre-} and \textit{post-collisional velocities} 
	 by the  following vectors in $\R^{d(s+k)}$ 
	\begin{align}
	\label{pre collisional vector}
	    V_{s+k} &  \  :=  \  ( \, V_s  \, ; \,  v_{s+1} \,  ,  \, \cdots \, ,  \, v_{s+k} \, )   \ ,    \\ 
	    \label{post collisional vector}
	    V_{s+k}^{* i } &  \  :=   \   ( \, v_1  \, , \,  \cdots \, , \,  v_i^* \, , \,  \cdots \, , \,  v_s  \,  ; \,   v_{s+1}^* \, ,\,  \cdots \, ,\,  v_{s+k}^* \, ) \   , 
	\end{align}
	where $(v_i^* , v_{s+1}^* , \cdots , v_{s+k}^*) \equiv  T_K^\omega ( v_{i} ,   v_{s+1} , \cdots, v_{s+1}^*) \in \R^{dK } $. 
	\vspace{2mm}
	 
	The operators that drive the BBGKY hierarchy then take the form (recall that  $K = k+1$)
	\begin{align}
	\label{bbgky operator}
	(
	\c_{s , s+ k }^N            \nonumber 
	f^{(s+k) }   
	)    
	(  V_s )  
	  = & 
	\frac{\beta_K \, N}{\binom{N}{K}}
	\binom{N \! -\! s}{K \! -\! 1}		 
	\sum_{ i =1 }^s 
	\int_{ \S_K \times \R^{d k }   }
	\Big(
	f^{(s+k) }
	(V_{s+k}^{* i } )
	-
	f^{(s+k)} (V_{s+k})             \nonumber 
	\Big)   
%	\\
%	& \times 
	  \d b_K(\omega )
	\d v_{s +1}
	\ldots 
	\d v_{s+k}        
		\\
 & 	\,  + \,   
	\calR_{s,s+k}^N \ . 
	\end{align}
 The operator $\calR_{s,s+k}^N$
	is a reminder term
	defined in \eqref{reminder}
	and whose explicit form we do not display here.
	Importantly, we will also show in Section \ref{section proof derivation} 
	that the operators $\c_{s,s+k}^N$
	satisfy Condition \ref{condition 1}.

	\vspace{2mm}

	In Section \ref{section proof derivation},
	we will show that 
	the operators $\c_{s,s+k}^N$
	given by   \eqref{bbgky operator}
	converge as $ N \rightarrow \infty $ 
	to the   operators
 $\c_{s,s+k}^\infty :
 L^1_{\textnormal{sym}} 
 (\R^{d(s+k)}) 
 \rightarrow L^1_{\textnormal{sym}} (\R^{ds})$
 given by 
	\begin{align}
	\label{boltzmann operators}
( \c_{s , s+ k }^\infty f^{(s+k)} )  (  V_s )
& = 
  \beta_K 
  K 
\sum_{ i =1 }^s 
\int_{ \S_K \times \R^{d k }   }
\Big(
f^{(s+k) }
(V_{s+k}^{*i } )
-
f^{(s+k)} (V_{s+k})
\Big)  \d b_K(\omega )
\d v_{s +1}
\ldots 
\d v_{s+k}      \ , 
\end{align}
where 
$V_{s+k} $ 
and  
$V_{s+k}^{*i}$ are as in \eqref{pre collisional vector} and \eqref{post collisional vector}, respectively. 
We verify that these operators satisfy Condition \ref{condition 2} (see Lemma \ref{lemma estimates}) 
and, therefore, fit the abstract functional framework.

\vspace{2mm}

We are now ready to introduce the 
\textit{Boltzmann hierarchy} as the infinite hierarchy 
\eqref{infinite hierarchy equation}
with the  operators $\C_{s,s+k}^\infty$ 
given by \eqref{boltzmann operators}:  
	\begin{equation}
	\label{boltzmann hierarchy}
	\partial_t 
	f^{(s)}
	= 
	\c_{s,s}^\infty f^{(s)} 
	+
	\cdots
	+
	\c_{s,s+  m}^\infty f^{(s + m )} \ .
	\end{equation}

\vspace{1mm}

Our main result concerns the limit from the BBGKY to the Boltzmann hierarchy.

\begin{theorem}
[From BBGKY to Boltzmann]
\label{thm bbgky to boltzmann}
    Let $X^{(s)} 
    =
    L^1_{\textnormal{sym}} (\R^{ds})$. 
    Let 
    $\bF_N  $
    and 
    $\bF $
    be mild solutions to the BBGKY hierarchy \eqref{bbgky hierarchy} 
    and
    Boltzmann hierarchy \eqref{boltzmann hierarchy}, with initial data 
    $F_{N,0}  \in X_0 $
     and
     $F_0 \in X_0$, respectively. 
     Additionally, assume that 
     $ F_{N,0} \xrightarrow{pw} F_0$, and that 
     $ \sup_{ N \geq 1}  \| F_{N,0} \|_{ 0} < \infty$. 
     Then, $\bF_N$ converges in observables to $\bF$. 
\end{theorem}

 We prove Theorem \ref{thm bbgky to boltzmann} as
   a  corollary of Theorem \ref{thm convergence}; its proof
 can be found in  Section \ref{section proof derivation}.

	\subsection{The Boltzmann Equation}
We start this subsection by noting that 
the ansatz $(f^{\otimes s})_{s \in \N}$
is a solution of the Boltzmann hierarchy
\eqref{boltzmann hierarchy}
if  
$f \in C ( [0,T] ; L^1  (\R^d ) )$
solves the following non-linear equation 
	\begin{equation}\label{Boltzmann equation}
	\begin{cases} 
	\partial_t f =   \beta_1 Q_1 (f   ) + \cdots +    \beta_M Q_M (  f, \ldots, f   ) \\ 
	f( 0 , \cdot  )  = f_0 \in L^1(\R^d ) \,  , 
	\end{cases}
	\end{equation}
	where the collision operators 
	$Q_K: L^1(\R^d )^K  \rightarrow L^1 (\R^d ) $
	were defined in the Introduction; see \eqref{collision operators}. 
	Global well-posedness for the equation \eqref{Boltzmann equation} was studied in \cite{gamba} in a slightly different setting.
	In Section \ref{appendix wp},  we adapt their   proof  to our situation 
	and obtain the following result. 
	
	\begin{proposition}[Global Well-posedness]\label{prop global wp}
		For all $  f_0 \in L^1(\R^d)$ with $ \int_{  \R^{d} }
		f_0(v) \d v   =1 $ and $ \| f_0  \|_{L^1 }  \leq  1, $  there is a unique solution 
		$f  \in C^1 ( \R ,L^1 (\R^d ) )$ to the Boltzmann equation \eqref{Boltzmann equation}.
		In addition, $\int_{\R^d}  f(t,v ) \d v = 1 $ and
		$\|  f(t) \|_{L^1} \leq 1 $ for all  $t \in \R$. 
	\end{proposition}

Now we are ready to state our result concerning propagation of chaos for the Master equation \eqref{master eq}. 
Namely, we prove the following result  
	
	\begin{theorem}[Propagation of Chaos]
		\label{thm prop of chaos}
		Let $f_{N,0} \in L^1_{\mathrm{sym}}(  \R^{dN })$	be non-negative  and  normalized to unity
		\begin{equation*}
		     \|     f_{N,0} \|_{L^1} = \int_{\R^{dN }}  f_{N,0} (V) dV = 1 \ . 
 		\end{equation*}
		Further, assume  that its sequence of marginals 
		$ (   f_{N,0}^{ (s) }  )_{s  \in \N }$ 
		converges pointwise weakly to the tensor product
		$ (  f_0^{\otimes s }   )_{s \in \N } $,  for some $f_0 \in L^1(\R^d )$. 
		Let 
		$f_N(t)$
		 be the solution of the Master equation \eqref{master eq}, with initial data $f_{N,0}$. 
		Then,   for all    $t\geq0 $,     $s \in \N $ and $\vp_s\in  L^\infty (\R^{ds})$ it holds that 
		\begin{equation}\label{thm 2 eq 1}
		\lim_{N \rightarrow \infty }
		\langle   f_N^{(s)} (t, \cdot ) , \vp_s       \rangle 
		=
		\<  f(t , \cdot )^{\otimes s} , \vp_s    \> 
		\end{equation}
		where $f(t,v )$ is the solution of the Boltzmann equation \eqref{Boltzmann equation}, with initial data $f_0 $. 
	\end{theorem}

 \begin{remark}
     Since the solution of the Master equation $f_N (t)$ is the probability density function of a probability measure (see Section \ref{master equation section}) it holds that 
   	\begin{equation}
		       \|     f_{N } (t)  \|_{L^1} 
		= \int_{\R^{dN }}  
		f_{N} (t , V) dV 
		= 1 
  \qquad \forall t \in \R   
 		\end{equation}
 and similarly for its sequence of marginals $f_N^{(s)} (t)$ . 
 \end{remark}

	\section{Applications}
	\label{section applications}
	In this section, we describe a set of examples that fit the framework introduced  in 
	Section \ref{section intro}
	and further developed in Section \ref{section main results}. 
	Namely, they satisfy $\textbf{(H1)}-\textbf{(H3)}$,
	and Theorem \ref{thm bbgky to boltzmann}
	and \ref{thm prop of chaos}
	can be applied to each of those models. 
	Some of the examples we consider have already been studied in the literature, and we recover existing results (see  Example 1 and 2 below). 
	Example 3, on the other hand, is new.

	\vspace{1mm}

	The following formula is helpful when trying to verify  the symmetric condition \hhh. 
	Let us  	regard a  linear map $T : \R^{dK} \rightarrow \R^{dK }$  as a collection of blocks  $ T = [ T_{ij}  ]_{i,j = 1}^K$, 
	where each $T_{ij } : \R^d  \rightarrow  \R^d $ is linear. Then, it holds that 
	\begin{equation}\label{permutation formula}
	\big(
\sigma \circ T \circ  \sigma^{-1}
	\big)_{i,j} 
	= T_{ \sigma(i ) , \sigma(j)}    \ ,
	\qquad i,j = 1 , \ldots, K  \  , \ \sigma \in S_K  \ . 
	\end{equation}

	\subsection{Examples}
	
\textit{(1) Binary Collisions.}
		Let $K =2 $ and take $\S_K = \S_1^{d-1}$, the $(d-1)$-dimensional unit sphere. 
		The transformation law $T_B$ is then defined according to the formulae
		\begin{align}
		& v_1^* = v_1  + \langle \omega, v_2 - v_1 \rangle \, \omega  \\ 
		& v_2^* = v_2 - \langle \omega, v_2 - v_1 \rangle \,  \omega
		\end{align}
		for $\omega \in \S_1^{d-1}$.
		It is straightforward to verify that $T_B$ is an involution that conserves both energy and momentum. 
		Hence, $\textbf{(H1)}$ and $\textbf{(H2)}$
		are verified. 
		Furthermore, we may write in block  form 
		\begin{equation}
		T^\omega_B
		=
		\begin{pmatrix}
		\1_d -  \< \cdot, \omega   \>\omega & 	 \< \cdot, \omega   \>\omega  \\
		\< \cdot, \omega   \>\omega  & \1_d -  \< \cdot, \omega   \>\omega 
		\end{pmatrix} \ , 
		\quad \omega \in \S_1^{d-1 }  
		\end{equation}
		where $\1_d$ is the $d$-dimensional identity. 
		In particular, it follows that
		$(	T^\omega_B)_{11} 
		 = 
		 (	T^\omega_B)_{22}
		$
		and
		$(	T^\omega_B)_{12} 
		 = 
		 (	T^\omega_B)_{21}.
		$
	This observation, combined with Eq. \eqref{permutation formula},   implies that 
		$\sigma \circ  T^\omega_B  \circ \sigma^{ -1 } = T_B^\omega $ for any $\sigma \in S_2$,
		which in turn implies $\textbf{(H3)}$.

 		\vspace{1mm}
 		
 		\noindent \textit{(2) Kac's Toy Model.}
		In dimension $d =1 $, M.  Kac \cite{Kac 1956} originally considers    $\S_2 = ( - \pi , \pi)$ 
		and     the    transformation law  
		$(v_1 , v_2) \mapsto T_{toy}^\theta (v_1 , v_2)$ 
		determined by the matrix
		\begin{align}
		T^\theta_{toy}
		=
		\begin{pmatrix}
		\cos \theta & \sin \theta \\
		- \sin \theta & \cos \theta 
		\end{pmatrix}  \  , 
		\qquad \text{ for } \theta \in
		( - \pi ,  \pi ) 
		\  . 
		\end{align}
Since this is an isometry, it satisfies
$\textbf{(H1)}$. 
We now proceed to verify $\textbf{(H2)}$
and
$\textbf{(H2)}$.
To this end, 
we  calculate  that for   $\sigma  = ( 1 \ 2) \in S_2$   and $ \theta \in (- \pi,\pi )$ it holds 
		\begin{align}
		\sigma 
		\circ  
		T^\theta_{toy} 
		\circ
		\sigma^{-1 } 
		\ = \ 
		[ T^\theta_{toy}  ]^{-1 }
		\ = \ 
		\begin{pmatrix}
		\cos \theta &  - \sin \theta \\
		+ \sin \theta & \cos \theta 
		\end{pmatrix}  
		\ = \ 
		T^{-\theta}_{toy}    \ . 
		\end{align}
		Consequently, we find that 
		$
		\sigma 
	\circ  
	T^\theta_{toy} 
	\circ
	\sigma^{-1 } 
		= (T^\theta_{toy})^{-1}
		\neq  T_{toy}^\theta $  for general $\theta$.
		However, a  change of variables $\theta \mapsto - \theta$ shows that 
		$\hh$ and $\hhh$ are verified, 
		provided we consider an interaction kernel of the form $ \d b_2 (\theta) = f(\theta ) \d \theta$, where $f \geq 0$ is integrable and even $f(\theta) = f(- \theta )$. 
These are exactly the conditions considered originally by M. Kac \cite{Kac 1956}. 
\vspace{1mm}

	\noindent 	\textit{(3) Symmetric Collisions of Order $K$. }
		Consider  the set of scattering angles 
		\begin{equation}
		\S_K = \{   \omega  =( \omega_1 , \ldots, \omega_K) \in \R^{dK } \  |  \ \omega_1^2 + \cdots  + \omega_K^2 = 1    \}   
		\end{equation}
		endowed with a probability measure  of the form $b(\omega) \d \omega$, where $(b\circ \sigma )(  \omega)  =  b(\omega)$ for all $\sigma \in S_K$. 
		We consider  the transformation law  $T_K$ given by 
		\begin{align} \label{central transformation}
		v_{  i  }^* = v_{  i }  
		- 
		2 \sum_{ \ell = 1 }^K 
		\< \omega_\ell , v_\ell   \> \, \omega_i \, , \qquad  i \in \{  1 , \ldots , K 			 \} \ . 
		\end{align}
		A straightforward calculation shows that $T_K$ is an involution that  conserves energy.
		In addition, the block form representation $[T_K^\omega ]_{i,j} = \delta_{i,j}\1_d - 2 \< \omega_j , \cdot  \> \omega_i  $
		and Eq. \eqref{permutation formula}
		imply
		that for all $\sigma \in S_K$ it holds
		\begin{equation}
		\sigma \circ
		 T^\omega_K  
		 \circ
		 \sigma^{ -1 }
		=
		T^{\sigma ( \omega)}_K   \  , \qquad  \omega \in \S_K \ .
		\end{equation}
		Since the underlying probability measure is invariant under the change of variables  $\omega  \mapsto  \sigma^{-1 } \omega$, one verifies that Hypothesis $\hhh$ is satisfied. 
		Note,   that $T_K$  does not conserve momentum. 
		However, if the space   $\S_K$  is replaced by 
		\begin{equation}
		\S_K^{'} = \{   \omega  \in  \S_K    \, | \,    \omega_1 + \cdots + \omega_K = 0    \}   
		\end{equation}
		one may easily verify that conservation of momentum holds.

	\subsection{Other models}
	In our results, we always assume   that $\h-\hhh$ are satisfied. 
	We note that  there exist models in the literature that fail to satisfy at least one of these conditions and we give
	two such examples.
	However, 
	our methods can be adapted to cover theses cases. 
	
	\subsubsection{Bobylev-Cercignani-Gamba Model}
For $ K  \leq M $, 
suppose that one is given   scalar velocities $(v_1, \ldots, v_K) \in \R^{K}$. 
	In \cite{gamba}, the authors propose a model for economic games in which the particles (or players) undergo a transformation law $T_{a,b}$ of the form 
	\begin{equation}
	v_i^* = a v_i  + b \sum_{   j  \neq i   } v_j  \, , \qquad i \in \{ 1, \ldots ,  K  \}
	\end{equation}
	where the real-valued coefficients $a$ and $b$ are random variables  on  a probability space $(\Omega, \calF , \mathbb{P})$. 
	Note that even when $K=2$, this transformation fails to conserve energy unless the coefficients are heavily constrained; conservation of energy would force 
	$  | a^2 - b^2 |  =1$. 
	Note, however,       that the   relation     
	$ 		[ \sigma \circ   T_{a,b }   \circ  \sigma^{ -1 }   ]_{i,j}
	=
	[     T_{a,b}     ]_{ \sigma ( i )  ,\sigma ( j ) } 
	=
	[ T_{a,b}  ]_{i,j} 
	$
	implies that condition $\hhh$ is verified,  independently of the underlying probability space,  or the specific structure of the coefficients $a$ and $b$. 	
	
	One may still consider the situation in which $d(\omega) := |  \det T_{a,b}^\omega   | >0$, that is, 
	the case for which $T_{a,b}$ is invertible. 
We expect that analogous results to
	Theorem \ref{thm bbgky to boltzmann}
	and
	Theorem
	\ref{thm prop of chaos} can be proven, leading to a   Boltzmann equation \eqref{Boltzmann equation},
	 with a collisional operator given by
	\begin{equation}
	Q_K 
	( f_1, \ldots , f_K )  ( v_1  )
	:=  
	K
	\int_{\S_{K  } \times \R^{  ( K  -  1  )   }}  
	\Big(
	\frac{1}{d(\omega)}
	(   \otimes_{\ell =  1 }^K  f_\ell  )
	\big( [ T_{a,b} ]^{-1}   V    \big)
	-
	(   \otimes_{ \ell =  1 }^K  f_\ell  )				\nonumber   
	(   V   )
	\Big)      \d b_K( \omega )      d v_{2 } \ldots  d v_{K  }. 
	\end{equation}

	\subsubsection{Non-symmetric ternary collisions} 
	Let us focus on the ternary case $K = 3$, and consider the $(2d-1)$-unit sphere $\S_3 := \S^{2d -1 }$ with the usual surface measure $\d \omega $. 
	As noted in  \cite{AmpatzoglouPavlovic2019}, 
	the relevant transformation law $T_{ter}$ is defined as 
	\begin{align}
	& v_1^* = v_1  -  c ( v_1 , v_2 , v_3; \omega) (\omega_1 + \omega_2 ) \\ 
	& v_2^* = v_2 + c ( v_1 , v_2 , v_3; \omega) \,  \omega_1 
	\qquad 
	\qquad 
	c ( v_1 , v_2 , v_3; \omega) = \frac{\langle \omega_1, v_2 - v_1 \rangle + \langle \omega_2, v_3 - v_1 \rangle}{1 + \langle \omega_1, \omega_2 \rangle}
	\\
	& v_3^* = v_3  + c ( v_1 , v_2 , v_3; \omega) \,  \omega_2 
	\end{align}
	where $\omega = (\omega_1, \omega_2) \in \S^{2d-1}$. 
	Despite conserving energy and momentum, Hypothesis $\hhh$ is not satisfied for this model. 
	We expect, however, that our methods can be adapted to show that similar results
	hold true, leading to a Boltzmann equation with a collisional operator of the form 
	\begin{equation*}
	    Q_3 
	    =
	    Q_3^{(1)}
	    +
	    2
	    Q_3^{(2)}  , 
	\end{equation*}
	where, for $f \in L^1(\R)$ we have 
 \begin{align}
     Q_3^{(1)}
     [f] (v_1 )
     &  = \int_{\S_3 \times \R^2 }
     \Big(
     f(v_1^*) f(v_2^*) f(v_3^*)
     -
     f(v_1) f(v_2) f(v_3)
     \Big) \d b (\omega)  d v_2 d v_3 \ , 
     \\
      Q_3^{(2)}
     [f] (v_2 )
     &  = \int_{\S_3 \times \R^2 }
     \Big(
     f(v_1^*) f(v_2^*) f(v_3^*)
     -
     f(v_1) f(v_2) f(v_3)
     \Big) \d b (\omega)  d v_1 d v_3  \ . 
 \end{align}

	\section{The Master equation}
	\label{master equation section}

In order to accommodate higher order interactions among particles, in this section we construct a new Markov process. We are inspired by the pioneering work of  M. Kac \cite{Kac 1956}, where the author outlined the procedure for constructing the Markov process corresponding to binary interactions. Our Markov process then leads to the Master equation 
\eqref{master eq}.  

For simplicity of the exposition we   work with   Euclidean space $\R^{dN}$, 
instead of restricting ourselves to the     \textit{energy spheres}
\begin{equation}
\calE_N := \{ V  =(v_1, \ldots, v_N) \in \R^{dN} \, \, :  \, \,   | V | = \sqrt{N}    \} \ . 
\end{equation}	
Our methods can be easily adapted to incorporate restrictions to $\calE_N$ (since conservation of kinetic energy satisfied by the transformation law leaves the energy spheres invariant).

%Subsequently, we will be studying its associated forward Kolmogorov equation for the law of the process; it will be our starting point for the derivation of Eq. \eqref{boltzmann type equation}. 
\vspace{1mm}

First, we describe the heuristics behind constructing our Markov process. As noted above, we incorporate higher-order collisions given by the transformation law
 \eqref{transformation law}. Then we give a sketch of the mathematical details of its construction as a jump process.
 We refer the reader to Appendix \ref{appendix markov} for a brief review of the theory of Markov processes, 
 including the notation that will be extensively used in this section.

\vspace{1mm}

In order to construct the continuous time Markov process $\textbf{V}_N$ we will first construct the simpler discrete time process $\textbf{Y}_N$, where $\textbf{Y}_N(n)$ represents the state of our $N$-particle system after the $n^\text{th}$ collision. 
Recall that we fix $(\S_K, T_K, b_K)$ with $K = 1, \ldots, M$ as introduced in Section \ref{section intro}.
Fix positive parameters $\{ \beta_K \}_{  K  =1 }^M   $ that satisfy the normalization condition
\begin{equation}
\beta_1 + \cdots  +  \beta_M  = 1   \ . 
\end{equation}
%together with  positive functions   $ 0 \leq b_K  \in  L^1(  \S_K  )$, for  $ K   \in  \{ 2  , \ldots, M  \}$,
%that
%satisfy 
%\begin{equation}\label{kernel condition}
%\int_{\S_K } b_K (\omega ) \d \omega = 1 \,  
%\quad  \textnormal{ and } \quad 
%b_K(\sigma \cdot \omega) = b_K(\omega ) \
%\forall \sigma \in S_K  \ . 
%\end{equation}
Here, the parameters $\beta_K$ represent the probability that a given collision will be of order $K$.
Given the distribution of $\textbf{Y}_N(n)$ we obtain the distribution of $\textbf{Y}_N(n+1)$, the system after one collision, by following the steps,
\begin{enumerate} 
	\item Select $K  \in \{ 1 , \ldots, M \}$
	 with probability $\beta_K$. This determines the order of the system's next collision.
	\item Select which $K$ of the $N$ particles will undergo this collision by choosing an ordered index $(i_1 , \ldots,  i_K)$ uniformly from $\calI_K$. This choice has probability $ (K !)^{-1 } \binom{N}{K }^{ -1 }$.
	\item Select the impact parameter $\omega \in \S_K$ according to the law $\d b_K(\omega )$.
	\item Update the velocities as follows, $$\textbf{Y}_N(n+1) = T_{i_1 , \ldots,  i_K}^\omega (v_1 , \ldots, v_N),$$ where $T_{i_1 , \ldots,  i_K}^\omega $ is given by \eqref{T extension}.
\end{enumerate}
If we start with a given initial distribution $\textbf{Y}_N(0)$ of our $N$-particle system we can $\textit{formally}$ construct our process $\textbf{Y}_N$ completely by repeating the above steps.
\vspace{1mm}

To construct $\textbf{Y}_N$ $\textit{rigorously}$ we introduce a Markov transition function acting on $V \in \R^{dN}$ and a Borel set $B \in \mathscr{B}(\R^{dN})$
\begin{equation}\label{markov 1}
\mu_N 
(V , B )
:=
\sum_{ K =1 }^M  \beta_K 
\sum_{(i_1 , \ldots,  i_K) }
\frac{1}{  K !  \binom{N}{ K }   }
\int_{\S_K  }
\mathds{1}_{B} (   T^\omega_{  i_1 , \ldots,  i_K   }{ V  } ) 
	\d b_K(\omega )
\, , 
\qquad V
\in
\R^{dN}
\, ,
\  B \in \mathscr{ B }(
\R^{dN}	
)   
\end{equation}
whose (bounded) generator 
$
	P_N : C_b  (  \R^{dN}  ) \rightarrow C_b   (  \R^{dN } )
 $
 satisfies
\begin{equation} \label{PN}
	(P_N  \vp )(V) :=
	\int_{  \R^{dN } } \vp ( U ) \mu_N (V , d U ) 
	= 
	\sum_{ K  = 1  }^M   \beta_K 
	\sum_{(i_1 , \ldots,  i_K)}
	\frac{1}{  K !  \binom{N}{K }   }    
	\int_{\S_K  }
	\!  \   \ \vp (   T^\omega_{ i_1 , \ldots,  i_K  }{ V  } )  \, 
		\d b_K(\omega )
	\quad V \in 
\R^{dN }
  \, .
\end{equation}

Given  $f_{N,0}   \in \mathrm{Prob}( 
\R^{dN }	
  ) $, the space of probability measures on $\R^{dN}$, by Proposition $\ref{PropEK}$
we can find a probability space $(\Sigma , \mathscr{F} , \mathbb{P}  ) $
and a 
Markov chain $   \{ \textbf{Y}_N ( n )      \}_{  n   = 0 }^\infty     :   \Sigma    \times \N_0   \rightarrow   
\R^{dN }	
 $, whose transition function is $\mu_N  $ and whose initial law is determined by $f_{N,0 }$. 
In other words, it holds that  for all $  n  \in \N_0  $ and $B  \in \mathscr{ B } (  
\R^{dN }	
  )$
\begin{align}\label{markov 2}
\mathbb{P}[  \textbf{Y}_N( n +1 )  \in  B  | \textbf{Y}_N (0 )  , \ldots, \textbf{Y}_N( n  ) ] =   \mu_N (\textbf{Y}_N( n  ) , B  ) 
\quad \textnormal{and} \quad
\mathbb{P} [\textbf{Y}_N(0 )\in  B ] =f_{N,0}( B ) \, . 
\end{align}
By computing the one step transition probability for $\textbf{Y}_N$, it can be checked that $\mu_N$ given in \eqref{markov 1} is the correct transition function for our process $\textbf{Y}_N$. 

\vspace{1mm}

In order to introduce continuous time into our process, consider an independent Poisson process
$\{ M(t)\}_{t = 0 }^\infty $ with rate $N $  (see Definition \ref{poisson process} in Appendix \ref{appendix markov}) and define the  Markov process $\textbf{V}_N(t)$ as the \textit{jump process}
\begin{equation} \label{jump process}
    \textbf{V}_N(t) := \textbf{Y}_N ( M(t)) .
\end{equation}
In particular, it can be shown that this jump process corresponds to the transition semigroup $\{  T(t) \}_{t \geq 0 }$ whose (bounded) generator is 
\begin{equation} \label{generator}
\L_N
:=
N ( P_N   -  \mathrm{id}) \, : 
	 C_b( 
	 \R^{dN}
	 )  \rightarrow 
	 C_b( 
	 \R^{dN}
	 ),
\end{equation}
where $P_N$ is defined in $\eqref{PN}$. The reader is referred to Section 2.2 of \cite{EthierKurtz} for details.

	Our starting point for the derivation of the Boltzmann equation \eqref{boltzmann type equation} will be the 
	dynamics associated to the law of the process $V_N(t)$.
	More precisely, let us denote its law  by  $F_{N}(t,\cdot ) :=  \mathbb{P}  [  V_N(t) \in \cdot \,  ]  $.
	This is a     probability measure on $\R^{dN }$, invariant under permutations--the symmetric property being equivalent to the particles being indistinguishable. 
	We make the additional assumption that the initial data has a symmetric density  $  f_{N,0} \in L^1_{ \mathrm{sym} }(\R^{d N})$. 
	Consequently, $F_N $ has  a density  $f_N $ that evolves according to  the \textit{Master equation}
	\begin{equation}
	\label{master eq 2}
	\begin{cases}
	\partial_t f_N = \Omega f_N \\
	f_N( 0   ) = f_{N,0} \in 
	L^1_{\mathrm{sym}}(   \R^{dN} ) 
	\end{cases} \,  ,
	\end{equation}
	where the generator $\Omega: L^1_{\mathrm{sym}} (\R^{dN}) \rightarrow L^1_{\mathrm{sym}} (\R^{dN})$ is the bounded linear operator determined  by the formula 
	\begin{equation}\label{omega}
	\Omega \,  f 
	= N \!  
	\sum_{K=1 }^M
	\beta_K 
	\sum_{  i_1 \cdots i_K  }
	\frac{1}{  K !  \binom{N}{ K  }   } 
	\int_{\S_K   }
	\big( 
	f \circ  T^\omega_{  i_1, \cdots, i_K   }  - f  \big)  \, \d  b_K (\omega )    \ , 
	\qquad 
	f \in L^1_{\mathrm{sym}}(   \R^{dN} ) \ . 
	\end{equation}

	%Our starting point for the derivation of a Boltzmann-type equation 
	%will be the associated \textit{forward Kolmogorov equation} (or, \textit{Master equation})  
	%\begin{equation}  \label{forward}
	%\begin{cases}
	%\partial_t F_N = \L_N^* F_N \\
	%F_N( 0 , \cdot ) = F_{N,0} \in \mathrm{Prob}(   \R^{dN} )
	%\end{cases} \,  
	%\end{equation}
	%where  the generator  is the Banach space adjoint of $\L_N$. Eq. \eqref{forward} follows from the fact that $V_N(t)$ is a Markov process associated to the semigroup
	%$ (   \exp (t \L_N) )_{t = 0}^\infty$; see Eq. \eqref{transition}. 
	
	%\begin{remark}
	%	Given $\sigma \in S_N$  we denote its action on $L^1_\mathrm{sym}(\R^{dN})$  by  $ f_\sigma (V) = f (\sigma \cdot V)$. 
	%	In particular, one may verify that the generator of the Master equation satisfies  $\Omega  ( f_\sigma  ) = (\Omega f)_\sigma $. 
	%	Consequently, $\Omega$  preserves permutational symmetry, i.e.  it holds  $f_N(t) \in L^1_\mathrm{sym}(\R^{dN})$ for all $t \geq 0 . $  
	%\end{remark}
	
	\begin{remark} \textit{Relationship to the deterministic setting.}
	    The Liouville equation is the deterministic analogue of the Master equation \ref{master eq 2}. Furthermore, $N$ is chosen for the rate of the Poisson process $M(t)$ in \eqref{jump process} to ensure a constant number of collisions per unit time per particle in the limit $N \rightarrow \infty$ and is analogous to the Boltzmann-Grad scaling in the deterministic setting.
	\end{remark}

	\section{Proof of Theorem \ref{thm convergence} }
	\label{section proof convergence}
	Throughout this section, we  assume that the estimates contained in Condition \ref{condition 1} and \ref{condition 2}
	are satisfied, together with assuming that $T>0$ satisfies \eqref{time}. 
	First, we   introduce   some notation and prove some preliminary inequalities. 
	
	\vspace{1mm}
In what follows, we will be using the same notation introduced in Subsection \ref{subsec-functional}	
	For $s \in \N $ let us introduce the canonical projections
	\begin{equation}
	\pi_s :  X   =  
	\bigoplus_{ r \in \N }
	X^{(r)    } \longrightarrow   X^{(s)}
	\end{equation}
	defined for $F = (f^{(s)})_{s \in \N} \in X$
	as
	$ \pi_s(F) := f^{(s)}$. 
	In particular, in terms of the objects 
	$\bF : [0,T] \rightarrow X $ and
	$\bF_N : [0,T] \rightarrow X $, convergence of observables
	(see Definition \ref{definition 3})
	is equivalent to the following statement: 
	for all $	  s  \in \N $ and for all $   \vp_s \in  X^{(s)*}  $ there holds 
	\begin{equation}
	\lim_{N \rightarrow \infty }  \langle  \pi_s   \bF _N (t), \vp_s    \rangle   =  
	\< \pi_s   \bF(t), \vp_s    \>  
	\end{equation} 
	uniformly in $t \in [0,T]$. 
	
	\vspace{1mm}
	
Let $\bC^N , \bC^\infty : X \rightarrow X $ 
be the linear transformations
introduced in   \eqref{C N} and \eqref{C infinity}, respectively.
The introduction of the projections $(\pi_s)_{s \in \N }$ will be particularly useful for proving norm estimates for the $s$-th components of the \textit{iterated} powers of $\bC^\infty$ ($\bC^N$, resp.). Namely,  for the operators

\begin{equation}
     (\bC^\infty)^n 
    =
    \underbrace{\bC^\infty \circ \cdots \circ  \bC^\infty }_{n \text{ times }}
     \, , 
     \qquad 
     n \in \N \ . 
\end{equation}

	More precisely, the following lemma holds true.

	\begin{lemma}\label{lemma powers}
	(a)
		If  $\bC^N$ satisfies Condition \ref{condition 1}, then
		 for every $\ell  \in \N$,  $ s \in \N$ 
		 and $F \in X $ there holds 
		\begin{align}				  
		\textstyle 
		\|  \pi_s 
		\big[ 
		(\bC^N)^\ell 
		F
		\big] 
		\|_{  X^{(s)} } 
		\	\leq  \ 
		\sum_{ k_1 = 0}^m
		\cdots
		\sum_{k_\ell = 0}^m 
		s ( s+ k_1 ) &  \cdots (s + k_1 + \cdots   +  k_{\ell -1 })   \\ 
		&  \times R_{k_1} \cdots R_{k_\ell} 
		\|  \pi_{ s + k_1 \cdots + k_\ell  } F\|_{  X^{(s + k_1 \cdots + k_\ell)} }			 \ .	\nonumber
		\end{align}
		
		\noindent (b)
		If  $\bC^\infty $ satisfies Condition \ref{condition 2}, then
		 for every $\ell  \in \N$,  $ s \in \N$ 
		 and $F \in X $ there holds 
		\begin{align}				  
		\textstyle 
		\|  \pi_s 
		\big[ 
		(\bC^\infty)^\ell 
		F
		\big] 
		\|_{  X^{(s)} } 
		\	\leq  \ 
		\sum_{ k_1 = 0}^m
		\cdots
		\sum_{k_\ell = 0}^m 
		s ( s+ k_1 )  & \cdots (s + k_1 +  \cdots   +  k_{\ell -1 })   \\ 
		&  \times \rho_{k_1} \cdots \rho_{k_\ell} 
		\|  \pi_{ s + k_1 \cdots + k_\ell  } F\|_{  X^{(s + k_1 \cdots + k_\ell)} }			 \ . 	\nonumber
		\end{align}

	\end{lemma}

	\begin{proof}
	We shall only present a proof for \textit{(b)}; that of \textit{(a)} is identical. 
		In what follows, we omit the subscript $X^{(s)}$ from the norms $\|  \cdot \|_{X^{(s)}}$.
		The proof goes by induction on $ \ell  \in \N$. Indeed, for $ \ell  =1 $ let $s \in \N $, $F = (f^{(s)})_{s \in \N}  \in X $  and estimate using Condition \ref{condition 2}
		that 
		\begin{align}
		\|   \pi_s \big[ \bC^\infty F \big]		\|  
		\ = \ 
		\| 					\nonumber  
		\C_{s,s}^\infty 
		f^{(s)}
		+
		\cdots
		+
		\C_{s, s+m }^\infty 
		f^{(s+m )}
		\|  
		&   \ \leq \  
		s				\nonumber 
		\big(
		\rho_0   \| f^{(s)} \|  
		+ 
		\cdots
		+
		\rho_m   \| f^{(s+ m  )}  \|  
		\big) \\ 
		&  \  =   \ 
		s  \sum \rho_k 
		\|    \pi_{s + k }    \,     F    \|   \textstyle 
		\label{estimate in proof lemma 1} \, .
		\end{align}
		Assume now that the result holds up to $ \ell  \in \N$. Then, for $s \in \N$ and $F  \in X $ we have that 
		\begin{align}
		\| \pi_s 
		\big[ 
		(\bC^\infty)^{  \ell  +1 } F 
		\big] 
		\|   \nonumber 
		&     \ \leq \ 
		\sum_{ k_1 = 0}^m
		\cdots
		\sum_{k_\ell = 0}^m 
		s ( s+ k_1 ) \cdots (s + k_1 + \cdots    +  k_{\ell -1 })   \\ 
		&    \qquad \qquad  \times \rho_{k_1} \cdots \rho_{k_\ell} 
		\|  \pi_{s + k_1 \cdots + k_\ell  }  \bC^\infty F\|  	 \ , 		\nonumber \\ 
		&    \ \leq \ 
		\sum_{ k_1 = 0}^m
		\cdots
		\sum_{k_\ell = 0}^m  											\nonumber 
		s ( s+ k_1 ) \cdots (s + k_1 + \cdots    +  k_{\ell -1 })   \\ 
		&   \qquad   \times \rho_{k_1} \cdots \rho_{k_\ell} 
		\sum_{k_{\ell+1}=0}^m
		(s + k_1 + \cdots k_\ell)
		\rho_{k_{\ell + 1 }}
		\|   \pi_{s + k_1 + \cdots k_{  \ell + 1 }} F  \|  \ . 
		\end{align}
		This finishes the proof of the lemma after elementary manipulations.  
	\end{proof}
	
	\vspace{1mm}
	
	The following lemma    will be useful throughout the proof of convergence.
	We recall that the well-posedness time $T_*$ was defined in 
	\eqref{time}. 
	
	\vspace{2mm}
	
	\begin{lemma}\label{elementary bounds}
	Let $s\in \N$,  $   \mu \geq -1      $ and
		 let $n \geq 10$.

\noindent 		(a)
If $\C^N$ satisfies Condition \ref{condition 1}, then for all $F \in X_\mu  $ there holds 
			\begin{equation}
		\textstyle 
		\|   \pi_s  
		\big[ 
		(   \bC^N)^n  F
		\big] 
		\|_{X^{(s)} }
		\    \leq  \ 
		\,   
		s e^{-\mu s }
\, 	n!  \,   (m T_*^{-1})^n 
		\, 
		( e \, n)^{s/m }  
\|F \|_\mu  \ . 
		\end{equation}

	\noindent (b)
		If $\C^\infty$ satisfies Condition \ref{condition 2}, then for all $F \in X_\mu $ there holds 
		\begin{equation}
		\textstyle 
		\|   \pi_s   
		\big[  
		(   \bC^\infty)^n  F 
				\big] 
		\|_{X^{(s)} }		\    \leq  \ 
		\,   
		s e^{-\mu s }
\, 	n!  \,   (m T_*^{-1})^n 
		\, 
		( e \, n)^{s/m }  
\|F \|_\mu  \ . 
		\end{equation}
	\end{lemma}
	
	\vspace{2mm}
	
	\begin{proof}
	Similarly as before, we shall only present a proof of \textit{(b)}. 
	Let $s,n,\mu$ be as in the statement of the lemma, and for the sake of the proof let us denote $\alpha = s/m$. 
	Then, 
		for any $0 \leq k_1 , \cdots , k_n \leq m  $
		we have the following upper bound 
			\begin{align}
			\label{lemma 5.2 eq 1}
		s (s + k_1 ) \cdots 
		(s + k_1 + \cdots +   k_{n-1 } ) 
		 & \leq     \nonumber 
		s (s + m )
		\cdots 
		(s + (n-1 )m ) \\
		 & 
		 =                  \nonumber
		 s m^{ n-1 } 
		(s m^{-1 }+ 1 )\cdots 
		(   sm^{-1 }  + (n -1 )  )  \\
			 & 
		 =
		 s m^{ n-1 }  
		( \alpha + 1 )\cdots            \nonumber         
		(   \alpha  + (n -1 )  )  \\
		 & 
		 =
		 s m^{ n-1 } (n-1 )!    \nonumber
		( \alpha + 1 )\cdots 
		\Big(    \frac{\alpha }{n -1 }  +  1  \Big) \\
		& \leq                      
		s\, m^n n! 
		( \alpha + 1 )\cdots 
		\Big(    \frac{\alpha }{n }  +  1  \Big)
		\end{align}
For notational convenience, we have replaced $n-1 $ by $n$; since we are only interested in the asymptotic behaviour when $ n \rightarrow \infty$,  such replacement is harmless. 
Next,  	using the fact that $\log (1 + x ) \leq x $ for all $x \geq   0 $
	one finds 
	\begin{align}
	    	( \alpha + 1 )\cdots 
		\Big(    \frac{\alpha }{n  }  +  1  \Big) 
	 & 	=               \nonumber
		\exp \log \Bigg(
			( \alpha + 1 )\cdots 
		\Big(    \frac{\alpha }{n }  +  1  \Big) 
		\Bigg)  \\
		&  = 
			\exp   \Bigg(
			\log 
			( \alpha + 1 ) + \cdots +           \nonumber
	 \log 	\Big(    \frac{\alpha }{n  }  +  1  \Big) 
		\Bigg)  \\ 
		& \leq 
			\exp   \Big(
		 \alpha  \big( 
		 1 + 1/2 + \cdots + 1/n 
		 \big)
		 \Big)
		 \ . 
	\end{align}
	For $  n \geq 10 $
	one has the standard bound
	$  \sum_{j =1 }^n 1/j \leq \log(n) + 1$. Consequently, we find
	\begin{align}
	\label{lemma 5.2 eq 2}
	    	( \alpha + 1 )
	    	\cdots 
		\Big(    
		\frac{\alpha }{n}  +  1 
		\Big) 
	& \leq 
\exp\Big(
\alpha  \log (n ) + \alpha 
\Big)
	=
	( e \, n )^{s/m } \ .
	\end{align}
		Next, we use the definition of the norm $ \|  \cdot  \|_\mu$ (see \eqref{mu norm}) 
		to find that 
		\begin{equation}
		\label{lemma 5.2 eq 3}
		\|  \pi_{s + k_1 + \cdots + k_n  } F  	\|_{ X^{(s + k_1 + \ldots +   k_n )} }
		\leq 
		\exp \big( -\mu 	(s + k_1 + \cdots + k_n)	\big) \|  F \|_\mu \ . 
		\end{equation}
		Hence, 
		by 		Lemma \ref{lemma powers}
and 
		\eqref{lemma 5.2 eq 1}, \eqref{lemma 5.2 eq 2}, \eqref{lemma 5.2 eq 3} 
	we find
		\begin{align}
		\textstyle 
		\|   \pi_s   (   \bC^\infty)^n  F \|_{ X^{(s)} }
	 	\   \leq     \      \nonumber 
	 &
		\,   
		s 
		\, 
		m^{ n }
		n! 
		\, 
		( e \, n)^{s/m }  
		\\
	 & \times 
	 \sum_{k_1 \ldots k_n}^m
	 \exp \big(  -\mu   (s + k_1 + \cdots + k_{n-1} )  \big)
		\rho_{k_1} \cdots \rho_{k_n} \, \| F \|_{\mu}  
		\end{align}
		from which   the desired estimate follows after elementary manipulations, taking into account the definition of the well-posedness time $T_*$--see \eqref{time}. 
	\end{proof}

 We are now ready to give a proof of Theorem \ref{thm convergence}.

	\begin{proof}[Proof of Theorem \ref{thm convergence}]
		Let $\bF_N = (f_N^{ (s)  })_{s \in \N } \in \bX_\bmu$ and 
		$\bF = (f^{(s)})_{s \in \N }\in \bX_\bmu $ 
		be as in the statement of Theorem \ref{thm convergence}
		with initial data
		$F_{N,0} = ( f_{N,0}^{ (s) } )_{s\in\N }\in X_{0}$ and
		$F_0 = (f_0^{(s)})_{s \in \N } \in X_0 $, respectively. 
		Recall that existence and uniqueness of mild solutions of both hierarchies is guaranteed by Proposition \ref{corollary 1} and Proposition \ref{corollary 2}.
		The idea of the proof is as follows: for fixed $s\in \N $,  starting from   both the finite and the  infinite hierarchy in mild formulation, we iterate the integral formulas  $ n \in \N $ times. 
		Next, we show that the initial conditions match in the limit $ N   \rightarrow \infty $ and the integral remainder term vanishes as $n     \rightarrow  \infty $, uniformly in $N$. 
		
		\vspace{1mm}
		
		Let us be more precise. First, we write the mild formulation of the solutions of both hierarchies 
		\begin{align}
		\bF_N (t) & = F_{N,0} + \int_0^t \bC^N   \bF_N(\tau ) \d \tau  \, ,  \\
		\bF (t) & = F_{0} + \int_0^t \bC^\infty   \bF (\tau ) \d \tau  \, .
		\end{align}
		Next, let us fix $s \in \N$  and iterate $n$ times the above equations to get
		\begin{align}
		\bF_N (t) &  \ = \ 
		\sum_{ \ell  = 0}^n 
		\frac{t^\ell }{\ell !}
		(\bC^N)^\ell  F_{N ,0} 
		\ + \ 
		\int_0^t \cdots \int_0^{t_{n  }}
		(\bC^N )^{ n + 1 }
		\bF_N(t_{ n + 1 }) \, \d t_{n+ 1} \cdots \d t_{1} \\
		\bF  (t) & \  =  \ 
		\sum_{ \ell  = 0}^n 
		\frac{t^\ell }{\ell !}
		(\bC^\infty )^\ell  F_{0} 
		\ + \ 
		\int_0^t \cdots \int_0^{t_{n  }}
		(\bC^\infty  )^{ n + 1 }
		\bF(t_{ n + 1 }) \, \d t_{n+ 1} \cdots \d t_{1} 
		\end{align}
		Once we project with $\pi_s$ and consider the pairing with 
		$\vp_s \in X^{(s)*}$
		, we note that   the contribution to this difference arises due to two terms:
		\begin{align}
		\big| 
		\langle \pi_s   \bF _N (t), \vp_s    \rangle  -  
		\< \pi_s   \bF(t), \vp_s    \> 
		\big| 
		\ \leq  \ 
		\calS_{N,n}(t) 
		\    +  \  
		\calI_{N,n}(t)
		\end{align}
		where $\calS_{N,n}(t)$ is the \textit{sum} given by 
		\begin{equation}
		\calS_{N,n}(t) 
		\  :=  \  
		\sum_{\ell  = 0}^n 
		\, \frac{t^\ell }{\ell !} \, 
		\big| 
		\langle \pi_s    (\bC^N)^\ell  F_{N ,0}     , \vp_s    \rangle  -  
		\< \pi_s   (\bC^\infty )^\ell  F_0, \vp_s    \> 
		\big| 
		\end{equation}
		and   where $\calI_{N,n}(t)$ is an \textit{integral remainder} term defined as 
		\begin{equation}
		\calI_{N,n} 
		(t) : = 
		\int_0^t 	\!  	\cdots  \! 			 \int_0^{t_{n  }}
		\bigg( 
		\|  \pi_s  (\bC^N  )^{ n + 1 }
		\bF_N    (t_{ n + 1 })  \|_{ X^{(s)} }   
		+
		\|   \pi_s  (\bC^\infty  )^{ n + 1 }
		\bF(t_{ n + 1 }) \, 
		\|_{ X^{(s)} }
		\bigg) 
		\d t_{n+ 1} \cdots \d t_{1} \,  , 
		\end{equation}
		and we assume without loss of generality that $ \|  \vp_s \|_{X^{  (s)* }} \leq 1 $.  
		We study these two terms separately.

		\vspace{2mm}

		\noindent \textsc{Integral Remainder Terms  $\calI_{N,n}$.} It  suffices to estimate the time  integrals, with respect to $n$, uniformly in $N$.
		We actually show that each integral, separately, converges to zero in $X^{(s)}$ norm, once we project via the map $\pi_s$. 
				Since the estimates are identical for  $\bF_N (t)$
	and  $\bF(t)$, we   only present  a proof for the latter. 
		
		%First, we start by noting that the following estimates follow from
		%Proposition \ref{corollary 1} and \ref{corollary 2}: for all $t \in [0,T]$ and $s \in \N $ there holds 
		%\begin{align}\label{estimates 1}
		%& \|   \pi_s  \bF (t) \|_{L^1 }
		% = \| f^{(s)}(t)  \|_{L^1 }
		% \leq e^{  - \bmu (t) s   } \| \bF \|_{\bmu_T}
		% \leq 2 e^{  - \bmu (t) s   } \| F_0 \|_{  0}  \\ 
		%&  \|   \pi_s    \bF_N (t) \|_{L^1 }
		% = \|    f_N^{(s)}(t)  \|_{L^1 }
		% \leq e^{  - \bmu (t) s   } \|    \bF_N   \|_{\bmu_T}
		% \leq 2 e^{  - \bmu (t) s   } \|  F_{N,0 } \|_{ 0}
		%\end{align}

 \vspace{2mm}

First, we introduce the following notation, convenient for estimating the nested integrals: 		
$$ \d \bar t_{n+1} \equiv  \d t_{n+1 } \cdots \d t_1 \ , \qquad  n \in \N \ . $$
Further, we recall that in Section \ref{section main results} we have introduced the function
$$
\bmu(t) = - t /T \ , \qquad t \in [0,T]
$$
		where $T <   m^{-1} T_* $, see \eqref{time}.
In view of  Lemma  \ref{elementary bounds} 
we find that,
 for all $ n \in \N $ and $t_{n+1} \leq t   \leq  T $, 
 the following estimate holds  
		\begin{align}
		    \|   
		\pi_s  
		(\bC^\infty  )^{ n + 1 }
		\bF(t_{ n + 1 })
		\|_{ X^{(s)} }	
		\leq 
			s 
			e^{-\bmu(t_{n + 1 }) s }
\, 	( n  + 1 )!  \,  
(m T_*^{-1})^{n+1} 
		\, 
		[ e   (n+1) ]^{s/m }  
\|\bF(t_{n+1}) \|_{ \bmu (t_{n+1}) }  \ . 
		\end{align}
		Consequently, we find
		\begin{align}
		    	\int_0^t 
		\! \cdots \!  
		\int_0^{t_{n  }}
		\|   
		\pi_s  
		(\bC^\infty  )^{ n + 1 }
		\bF(t_{ n + 1 })
	 \|_{ X^{(s)} }	        
	   	\d \bar t _{n +1 }   
	    & 	\leq  
			  	 	s           \nonumber
		   	( n  + 1 )!  
(m T_*^{-1})^{n+1} 
		\, 
	  	[ e   (n+1) ]^{s/m }  \\ 
		   &  \qquad \qquad \qquad  \times 
		   \int_0^t 
		\! \cdots \!  
		\int_0^{t_{n  }}                \nonumber
		 	e^{-\bmu(t_{n + 1 }) s }
\| \bF (t_{n+1})\|_{ \bmu (t_{n+1}) }   
\d \bar t _{n +1 }   \  ,  \\ 
		  & 	 \leq   	s 
		   	( n  + 1 )!             \nonumber
(m T_*^{-1})^{n+1} 
		\, 
		[ e   (n+1) ]^{s/m }  \\
		   & \qquad \qquad \qquad  \times            \nonumber
		   \| \bF \|_{ \bmu }   
		   	\int_0^t 
		   	\! \cdots \!  
		\int_0^{t_{n  }}
			e^{-\bmu(t_{n + 1 }) s }
\d \bar t _{n +1 }                \ ,      \nonumber
      \\          
			 & 	  \leq    	s 
		   	( n  + 1 )!  
(m T_*^{-1})^{n+1}                                  \nonumber
		\, 
		[ e   (n+1) ]^{s/m }   \\
		   	 &\qquad \qquad  \qquad \times \|\bF \|_{ \bmu }           \nonumber
		 e^s  \frac{T^{ n + 1 }}{(n + 1)! }     \      ,      \\
		   & =  
		   	s    (m T T_*^{-1} )^{n+1}                      
		[ e   (n+1) ]^{s/m }  
		e^s 
				   \| \bF \|_{ \bmu }   
		\  . 
		\end{align}
			We recall that $T$ was chosen small enough in Eq. \eqref{time}
		so that   $  m T T_*^{-1} <1  $ holds true. Therefore, as $ n \rightarrow \infty$, the integral remainder term vanishes.

		\vspace{2mm}

		\noindent \textsc{Controlling the Sum $\calS_{N,n}$.} 
		First, we show  that the following result holds. 
		
		\vspace{1mm}
		
		\begin{lemma}
			Let $F_{N} \in X$ converge pointwise weakly to $F \in X$, 
			and let $\bC^N$ converge to $\bC^\infty $ in the sense of Definition \ref{definition convergence}.  
			Then,  for all $ \ell  \in \N$ it holds that $(\bC^N)^\ell  F_N$ converges pointwise  weakly to $(\bC^\infty)^\ell  F$. 
			In other words, for all $s \in \N$, $ \ell  \in \N$ and 
			$\vp_s\in  { X^{(s)* } }	 $
			it holds that 
			\begin{equation}
			\lim_{N \rightarrow \infty }	\nonumber 
			\langle
			\pi_s \big[     (\bC^N)^\ell  F_{N }   \big]   , \vp_s    
			\rangle  =
			\< \pi_s   \big[   (\bC^\infty )^\ell  F     \big] , \vp_s    \>  \, .
			\end{equation}
		\end{lemma}
		
		\vspace{1mm}
		
		\begin{proof}
			The proof goes by induction on $ \ell  \in \N$. 
			The case $ \ell  =1 $ follows from the definition of convergence from $\bC^N$ to $\bC^\infty$. 
			Assume now that the result holds for   $ \ell  \in \N $, i.e.  $G_N = ( \bC^N)^\ell  F_N$ converges weakly to $G = (\bC^\infty)^\ell  F$. 
			It follows that $\bC^N G_N$ converges pointwise 
			weakly to $ \bC^\infty G$. This finishes the proof. 
		\end{proof}

		\noindent \textsc{Conclusion.}
		First, we take the limit $N\rightarrow \infty$.
		Namely we  put our two estimates together to find that for all $ n  \geq 1$ there holds 
		\begin{align}
		\limsup_{N \rightarrow \infty }
		\big| 
		\langle \pi_s   \bF _N (t), \vp_s    \rangle  -  
		\< \pi_s   \bF(t), \vp_s    \> 
		\big| 
		& \  \leq  \ 
		\limsup_{N \rightarrow \infty }\calS_{N,n}(t) 
		\    +  \  
		\limsup_{N \rightarrow \infty }\calI_{N,n}(t) \ ,  \\
		&  \ \leq  \ 
	 	s    (m T T_*^{-1} )^{n+1}                      
		[ e   (n+1) ]^{s/m }  
		e^s 
		\big(
		\| \bF  \|_{\bmu}
		+
		\sup_{  N \in \N }
		\| \bF_N  \|_{ \bmu }
		\big) \ .           \nonumber 
		\end{align} 
		Thanks to Proposition \ref{corollary 1}, one has    that
		$  
		\| \bF_N  \|_{ \bmu }
		\leq 
		(1 - \theta_2)^{-1} 
		\| 
		F_{N,0}
		\|_{0 } $
		for all $ N \geq 1$. 
		Thus, 
		$
		\sup_{  N \in \N } 
		\| \bF_N  \|_{  \bmu } <\infty 
		$
		due to our assumptions on the initial data. 
		The conclusion of the theorem now follows after we take the $ n \rightarrow \infty$ limit. 
	\end{proof}

	\section{Proof of Theorems \ref{thm bbgky to boltzmann} and \ref{thm prop of chaos}}
	\label{section proof derivation}
	Throughout this section,    $f_{N}  $ denotes the solution of the Master equation \eqref{master eq}, and  $(f_N^{(s)})_{s\in\N}$ denotes   its sequence of marginals, defined in \eqref{marginals}. We recall that these quantities are symmetric with respect to the permutation of their variables. 
	
	\subsection{Calculation of BBGKY}
	In what follows, we fix the number of particles $N \geq M$ and some order  $ s  \leq N $ of the marginals. We start with the following calculation
	\begin{align}
	\label{calculation eq 1}
	\partial_t f_N^{  (s) }   =   \partial_t \tr_{s+1, \ldots, N}  \big(  f_{N}   \big)  = \tr_{s+1, \ldots, N}  \big( \partial_t f_N    \big) = \tr_{s+1, \ldots, N}  \big(  \Omega f_{N} \big)
	\end{align}
	where we recall that $\Omega$ is the linear operator introduced in Eq.  \eqref{master eq}.
	Hence, due to   \eqref{calculation eq 1},  linearity of the trace map, and the definition of $\Omega$ it follows that 
	\begin{align}
	\label{derivation eq 1}
	\partial_t f_N^{(s)}  
	= 
	\sum_{K  = 1}^M   
	\beta_K
	\frac{N }{ K! \binom{N}{K}}
	\sum_{   i_1  \cdots i_K  }
	\tr_{s+1, \ldots, N}
	\big(  \Omega_{ i_1  \cdots  i_K } f_{N}  \big)
	\end{align}
	where for each $   K  = 1 , \ldots \, , M $ and $ ( i_1 , \cdots , i_K )  \in \calI(K)$, defined in \eqref{index set}, we have introduced the operator 
	\begin{align}
	\label{omega indices}
	\Omega_{ i_1 \cdots  i_K } f  
	= 
	\int_{\S_K   }
	\big( 
	\, f \circ     T^\omega_{  i_1  \cdots   i_K			  }    - f  \, 
	\big)  
	\, \d b_K(\omega )
	\, , \quad   f  \in 
	L^1(\R^{dN } )
	\ . 
	\end{align}
	Thus, it remains to calculate  the quantity  $\tr_{s+1, \ldots, N}  (\Omega_{ i_1 \cdots i_K  } f) $ for arbitrary  $ (   i_1 ,\ldots ,i_K ) \in \calI(K) $ 
	and
	$f \in  L^1_{\mathrm{sym}}
	(\R^{dN } )$. 
	
	\vspace{1mm}

	The first step in this direction is exploiting the symmetric condition given in  \eqref{symmetry condition} given in $\hhh$. 
	This is the content of the following lemma. Recall that $S_K$ stands for the group of permutations of $K$ elements.  
	
	\vspace{1mm}
	
	\begin{lemma}\label{lemma combinatorics}
		For all $K = 1  , \ldots, M $, $(i_1, \ldots, i_K) \in \calI(K)$ and $\gamma \in S_K$ it holds that 
		\begin{equation}\label{lemma combinatorics eq 1}
		\Omega_{   i_{\gamma(1)} \cdots i_{ \gamma(K )}    }   = 
		\Omega_{i_1 \cdots i_K } \ . 
		\end{equation}
	\end{lemma}

\begin{proof}
 We divide the proof into two steps. In the first one, we assume that the collection of indices is a permutation of 
 the first $K$ indices: $ \{ 1,\ldots, K \}.$   In the second step, we show how the general case follows from the particular one.

 \vspace{1mm}
 
 \textit{Step One.} 
 Let $\gamma \in S_K$  be any permutation of the elements $ \{ 1,\ldots, K \}$, and denote by $\Gamma = \gamma \times id_{N-K}$
 its natural extension to $S_N$. 
 Let $f \in L^1_{\textnormal{sym}} (\R^{dN}) \cap C(\R^{dN})$
 and denote by $\Omega^+ = \Omega + { id }  $ 
 the \textit{gain term} of Eq. \eqref{omega indices}.
 Then, we calculate that for all $V \in \R^{dN}$
 \begin{align}
[ \Omega^+_{1 \cdots K} f ] (V)
 & =
\int_{\S_{K  }}
f \big[ 
T^\omega_{1 \cdots K} 
V
\big] 
\, 
d b_K(\omega ) \\
& 
= \int_{\S_{K  }}
f \big[
T^\omega_K (v_1, \ldots, v_K) ; 
v_{K+1} , \ldots, v_N
\big]
\, 
d b_K(\omega ) \\
 & = \int_{\S_{K  }}                    \label{lemma 5.4 eq 1}
f 
\Big[ 
\big( 
  \gamma^{-1} \circ  T^\omega_K \circ    \gamma 
\big) 
(v_1, \ldots, v_K) ; 
v_{K+1} , \ldots, v_N
\Big]
\, 
d b_K(\omega ) \\
& = \int_{\S_{K  }}
f \Big[
 \big(
  \Gamma^{-1 } \circ (T^\omega_K \times id_{\R^{d(N-K)}} ) \circ \Gamma \big)
   \, V 
\Big] 
\, 
d b_K(\omega ) \\
& = \int_{\S_{K  }}
f \big[
T^\omega_{\gamma(1) \cdots \gamma(K) }
V 
\big]
\, 
d b_K(\omega ) \\
& = 
[ \Omega^+_{ \gamma ( 1 )  \cdots \gamma ( K ) } f ] (V)
 \end{align}
 where   we have used $\hhh$ to obtain \eqref{lemma 5.4 eq 1}. 
 Since $L^1_{\textnormal{sym}} \cap C$ is a dense subspace of $L^1_{\textnormal{sym}}$, this finishes the proof of the first step.

 \vspace{1mm}
 
 \textit{Step Two.} Let now $(i_1 , \ldots, i_K) \in \calI (K)$ be arbitrary, and consider $\gamma \in S_K$   and $\Gamma \in S_N$ 
 as in Step One. 
 First, we make a general observation: 
 for all   $\sigma \in S_N$, $f \in L^1_{\textnormal{sym}} (\R^{dN })$ and $ V \in \R^{dN}$    the following   identity holds
 for the associated gain term
 \begin{equation}
 \label{general obs}
  [  \Omega_{ \sigma (1) \cdots   \sigma (K) }^+ f    ] (V)
   = 
 \int_{\S_{K }}
 f   \big[ 
 \big(      \sigma  \circ   (T_K^\omega \times id_{ \R^{d ( N -  K ) } } )    \circ     \sigma^{-1 }     \big)  
 \,  V 
 \big] 
 d b_K(\omega )    
   = 
 [ \Omega_{1 \cdots K }^+ (f \circ    \sigma )  ]  (   \sigma^{-1 }   V)    \ . 
 \end{equation}
Consequently, the same identity holds for the full operator as well. 
Now, we choose $\sigma$ such that $\sigma(1) = i_1,  \,    \ldots, \,   \sigma(K) =i_K $. 
Then, Step One and the general observation imply that
\begin{align}
 [  \Omega_{i_1 \cdots i_K } f    ] (V)
& = 
 [ \Omega_{1 \cdots K } (f \circ    \sigma )  ]  (   \sigma^{-1 }   V)     \\
 & = 
  [ \Omega_{ \gamma ( 1 )  \cdots  \gamma ( K )  } (f \circ    \sigma )  ]  (   \sigma^{-1 }   V)     \\
  &= 
    [ \Omega_{   1   \cdots    K } (f \circ    \sigma  \circ \Gamma )  ]   \big(     (  \Gamma^{-1 } \circ \sigma^{-1 } ) \,   V  \big)      \\
& = 
 [  \Omega_{ \sigma (\gamma(1)) \cdots  \sigma ( \gamma (K))} f    ] (V) \ .
\end{align}
Since $ \sigma ( \gamma (\ell)) = i_{\gamma (\ell) }$ for all $\ell \in \{ 1 , \ldots , K\}$, the proof is complete.
 	\end{proof}

	We apply Lemma \ref{lemma combinatorics}   in order to get a simplified expression of $\Omega$. 
	More precisely, we obtain that for all $ K = 1 , \ldots , M$ it holds that 
	\begin{equation}
	\sum_{  i_1 \cdots i_K  }
	\Omega_{ i_1   \cdots    i_K   }  
	\ =  \! \! 
	\sum_{   i_1 < \cdots  <i_K  }
	\sum_{\mu \in S_K }
	\Omega_{   i_{\mu(1)} \cdots  i_{ \mu(K )}    }   
	\ = \ 
	K!  \! \!   \sum_{    i_1 < \cdots  <i_K   }
	\Omega_{   i_{ 1 } \cdots  i_{ K }    }    \ . 
	\end{equation}
	Consequently, we may plug this back in Eq. \eqref{derivation eq 1} to conclude that 
	\begin{equation}\label{eq 1}
	\partial_t f_N^{ ( s)  }   =  \sum_{K  = 1 }^M   \beta_K  \frac{N}{ \binom{N}{K}} 
	\sum_{      i_1 < \cdots <  i_K   } 
	\tr_{s+1, \ldots, N}  \big(   \Omega_{   i_1 \cdots i_K   }    f_{N}    \big)   \ . 
	\end{equation}
	Thus, it suffices to calculate $\tr_{s+1, \ldots, N}  (\Omega_{ i_1 \cdots i_K  } f) $ only for    ordered indices $   i_1  < \cdots < i_K  $ and symmetric functions $f \in L^1_{\mathrm{sym}} (\R^{dN }) $.
	The following family of operators  is defined with that purpose.

	\vspace{2mm}
	
	\begin{definition}
	\label{definition 5}
		Let $ K   = 1, \ldots,  M , \ $ $n =  1 , \ldots, K   $     and   denote      $r \equiv  K - n $.
		For all indices  
		$1 \leq i_1 < \cdots <  i_n    \leq  s  $ we define the operator
		$$ 
		C^{ s,K,n}_{i_1   \cdots i_n  } : 
		L_{\textnormal{sym}}^1( \R^{d (s  +  r  )}  ) \rightarrow L_{\textnormal{sym}}^1( \R^{ d s    }  )
		$$
		as follows:
		\begin{enumerate}
			\item 
			For $r = 0 $ we set 
			$C^{s,K,n}_{i_1 \cdots i_n  }    := \Omega_{   i_{ 1 } \cdots  i_{ K }    }$ . 
			\item For $ r \geq 1 $  and  $s + r \leq N $ and    we set 
			\begin{align}
			C^{s,K,n}_{     i_1 \cdots i_n  }  
			f^{(s + r )} (V_{s })
			: =
			\int_{\S_K \times \R^{dr} }  
			\big(
			f^{(s + r )}
			(
			 V^{*i_1 \cdots * i_n }_{s + r } &     
			)
			-
			f^{(s + r)}
			( V_{s + r } ) 
			\big) 	    	\d b_K(\omega ) 
			d v_{s + 1}
			\cdots 
			d v_{s + r } 		\nonumber 
			\end{align}
			where $V_s \in \R^{ds}$ ,   $V_{s + r }  \equiv (V_s, v_{s+1} , \ldots, v_{s+r}) \in \R^{d(s+r)}$ and
			\begin{equation}
			    V_{s+r}^{*i_1 \cdots * i_n } 
			    : = 
			    ( v_1 ,  \, \ldots \, , v_{i_1}^* ,  \, \ldots \, , v_{i_n}^*,  \, \ldots \, , v_{s }  \, ; \, v_{s+1}^* , \,  \ldots \, , v_{s+r}^* ) \in \R^{d(s+r)}
			\end{equation}
			with $ (  v_{i_1}^* , \ldots, v_{i_n}^*,   v_{s+1}^* , \ldots, v_{s+r}^*  ) = T^\omega_K  
			(v_{i_1} , \ldots, v_{i_n},  v_{s+1} , \ldots, v_{s+r} )  \in \R^{dK}$
			\vspace{1mm}
			
			\item For $ r \geq 1 $ and  $s +r > N$ we set $C^{s,K,n}_{i_1   \cdots i_n  } \equiv 0 $. 
		\end{enumerate}
	\end{definition}

\begin{remarks}
	A few comments are in order. 
	\begin{remarklist}
		\item 
		    It can be helpful to keep in mind that $s$ is the order of the marginal and $r$ is the number of interacting particles that get traced over by the operator $C^{s,K,n}_{i_1   \cdots i_n  }$.
		\item 
		$C^{s,K,n}_{i_1  \cdots i_n  }  $ is bounded with operator norm $ \| C^{s,K,n}_{i_1 \cdots i_n  }   \| \leq 2$. 
	\end{remarklist}
\end{remarks}

%	\begin{remark}
%	    It can be helpful to keep in mind that $s$ is the order of the marginal and $r$ is the number of interacting particles that get traced over by the operator $C^{s,K,n}_{i_1   \cdots i_n  }$.
%	\end{remark}
%	
%	\begin{remark}
%		$C^{s,K,n}_{i_1  \cdots i_n  }  $ is bounded with operator norm $ \| C^{s,K,n}_{i_1 \cdots i_n  }   \| \leq 2$. 
%	\end{remark}

	The following lemma is the main result concerning the operators just introduced.

	\begin{lemma}\label{prop identity}
		Let $ K   = 1, \ldots,  M , \ $ $n =  	1 , \ldots, K   $  and let $r \equiv K - n  $. 
		Assume that $ s + r \leq N $ and  consider $ K $  ordered indices such that 
		\begin{equation}
		1 \leq i_{1} < \cdots  < i_{n} \leq s < i_{{n+ 1 }} < \cdots < i_{n+r} =  i_{K} \leq N \, .
		\end{equation}
		Then, for all $ f \in  L^1_{\mathrm{sym}} (  \R^{dN})$ the following identity holds 
		\begin{equation}\label{lemma5  eq 1}
		\tr_{s+1, \ldots, N}  \big(\Omega_{ i_1 \cdots i_K } f   \big)  
		=
		C^{s,K,n}_{ i_{1} \cdots   i_{n }} [   \tr_{ s + r+1, \ldots, N } f  ] \, .
		\end{equation}
	\end{lemma}
	
	\vspace{1mm}
	\begin{remark}
		 The main consequence  of the previous result is that 
		the left hand side  of Eq. \eqref{lemma5  eq 1}  is independent  of the last $r$ indices      
		$( i_{n+1}, \ldots\, , i_{n+r }) $. 
	\end{remark}

 	\begin{proof}
 		Since $s+r \leq N $, there are two cases.
 	
 	\noindent \textit{(i)}. 		
 	Let $r=0$.  Then,  $i_1 < \cdots < i_K \leq s$. In particular,    all of the particles that are being traced out are not interacting.  	Consequently, it is easy to show that
 			\begin{equation}
\tr_{s+1, \ldots, N}  \big(  \Omega_{ i_1 \cdots i_K } f   \big)  
 			=
 			\Omega_{ i_1 \cdots i_K }  [   \tr_{ s + 1, \ldots, N } f  ]
 			=
 			C^{s,K,n}_{i_1 \cdots i_n  }  
 			[   \tr_{ s + 1, \ldots, N } f  ] \ .
 			\end{equation}
 			\vspace{1mm}
 			
\noindent \textit{(ii)}.  			Let $r \geq 1$. 
 			Let  $f \in   L^1_{\mathrm{sym}} (  \R^{dN})$ and fix $V_s \in \R^{ds}$. 
 			Let $\mu$ be any permutation of the elements $\{  s+1, \ldots, N\}$. 
 			Then, we may implement the change of variables 
 			$(v_{s+1}, \ldots, v_N)  \mapsto 
 			(
 			v_{\mu^{-1} (s+1)  } , \ldots, v_{ \mu^{-1}(N) }
 			)$
 			in the following expression 
 			\begin{align}
 			\label{lemma 6.2 eq 1}
 			\int_{\R^{d(N-s )}}
 			f \big[ T^\omega_{i_1 \cdots i_K}
 			 (
 			 V_s ; v_{s+1} , \ldots, v_N 
 			 &
 			 ) 			\nonumber 
 			  \big]
 		  	\d v_{s+1} \ldots, \d v_{ N } \\ 
 		 & 	=
 			\int_{\R^{d(N-s )}}			\nonumber 
 			f 
 			\big[
 		  T^\omega_{i_1 \cdots i_K}
 			  (V_s ; v_{\mu^{-1} (s+1)  } , \ldots, v_{ \mu^{-1}(N) })
 			   \big]
 			\d v_{s+1} \ldots, \d v_{ N }   \\
 			& = 
 				\int_{\R^{d(N-s )}}
 			f \big[ 
 		 \big( 
 		 	T^\omega_{i_1 \cdots i_K}  \circ  \, (  id_{s}  \! \times \!   \mu^{-1} ) \,   \big)
 			(
V_s ; v_{s+1} , \ldots, v_N 			
 			)
 			 \big]
 			\d v_{s+1} \ldots, \d v_{ N }   
 			\end{align}
 			where we recall that we identify $id_{s} \times  \mu^{-1} $ with its group action over $\R^{dN}$, 
 			i.e.   we write
 		  $$
 		  (id_{s} \times  \mu^{-1} )(V_s ; v_{s+1} , \ldots, v_N 			) 
 		  =
 		  (V_s ; v_{\mu^{-1} (s+1)  } , \ldots, v_{ \mu^{-1}(N) }) . 
 		  $$
 		  Next, since $f \in L^1_{\mathrm{sym}} $, there holds $f = f \circ \bar \mu $,
 		  where we denote $\bar \mu  \equiv  id_s  \! \times \!  \mu  \in S_N$. 
 			Therefore, denoting $V \equiv  (V_s, v_{s+1}  , \ldots, v_N )$, 
 			we obtain thanks to \eqref{lemma 6.2 eq 1} and permutational symmetry that
 			\begin{align}
 				\int_{\R^{d(N-s )}}
 			f \big[ T^\omega_{i_1 \cdots i_K}
 		V 
 			\big]
 		 	\d v_{s+1} \ldots, \d v_{ N } 
 		  & 	=
 		 		\int_{\R^{d(N-s )}}
 		 	f \big[ (  \bar \mu \circ T^\omega_{i_1 \cdots i_K} \circ \bar \mu^{-1}   	) \, V  
 		 	\big]				\nonumber 
 		 	\d v_{s+1} \ldots, \d v_{ N }  
 		 	\\
 		 	 & = 
 		 		\int_{\R^{d(N-s )}}
 		 	f \big[  T^\omega_{ \bar \mu (i_1)  \cdots  \bar \mu (i_K ) }   V  
 		 	\big]
 		 	\d v_{s+1} \ldots, \d v_{ N }  
 			\end{align}
where the last line follows from the definition of $T_{i_1 \ldots i_K}^\omega$ (see \eqref{T extension}) 
upon conjugation with $\bar \mu $. 
Since $1 \leq i_1  < \cdots < i_n \leq s, $ we must have $\bar \mu (i_\ell) = i_\ell$ for $ 1 \leq \ell \leq n$. 
 	Further, since   $s +1   \leq  i_{n+1} < \cdots < i_{n + r }   \leq N $, 
 	we may choose $\mu$ such that $\mu (i_{n+1}) = s+1 ,  \ldots,  \mu(i_{n+r}) = \mu (i_K) = s+r$.
 	Consequently, we find
 	\begin{align}
 		\int_{ \R^{d(N-s )}}
 f \big[ 
 T^\omega_{i_1 \cdots i_K}
 V 
 \big]
 \d v_{s+1} \ldots, \d v_{ N } 
 & 	=
 \int_{\R^{d(N-s )}}
 f \big[  
 T^\omega_{ i_1 \ldots i_n , s+1 \cdots s +r  } V
 \big]			  
 \d v_{s+1} \ldots, \d v_{ N }  
 	\end{align}
Next, using the notation introduced in Definition \ref{definition 5}, 
we   are able to write  
\begin{equation}
    T^\omega_{ i_1 \ldots i_n , s+1 \cdots s +r  } V = (V_{s+r}^{*i_1 \cdots * i_n } , v_{s+r +1 } , \ldots, v_N) \ . 
\end{equation} 
Hence, we may use Fubini's theorem over the space 
$\R^{d(N-s ) }  = \R^{dr} \times \R^{d (N -s -r )}$ 
to find that
 			\begin{align}
 		\int_{ \R^{d(N-s )}}
 		f \big[ 
 		T^\omega_{i_1 \cdots i_K}
 		V \nonumber
 		\big]
 	 & 	\d v_{s+1} \ldots, \d v_{ N }  \\ 
 		& 	=
 		\int_{\R^{d(N-s )}}
 		f \big[  
 	 V_{s+r}^{*i_1 \cdots*  i_n }  , v_{s+r +1  } , \ldots, v_{N}
 		\big]				\nonumber 
 		\d v_{s+1} \ldots, \d v_{ N }   \\ 
 		&  = 
 		\int_{\R^{ dr }}
 		\Big(
 		\int_{\R^{d (N - r - s )}}
 			f \big[  
 		V_{s+r}^{*i_1 \cdots * i_n }  , v_{s+r +1  } , \ldots, v_{N}
 		\big]		\,
 		\d v_{s+ r +1 } \cdots \d v_{N}
 		\Big) 
 		\, \nonumber
 		\d v_{s +1 }
 		\cdots
 		\d v_{s + r } \\ 
 		& = 								 			\label{eq 1 lemma 6.2}
 		 		\int_{\R^{ dr }}
 		 		f^{(s+ r )} [V_{s+r}^{*i_1 \cdots*  i_n } ]
 		 			\, 
 		 		\d v_{s +1 }
 		 		\cdots
 		 		\d v_{s + r } 
 		\end{align}
 		where, in order to obtain the  last line, we have used the definition of the marginals introduced in Section \ref{section main results}. 
 	Similarly,  one can prove  that 
 		$\int_{\R^{d (N  - s)}}   f  [  V ] \d v_{s+1} \ldots \d v_{N}
 		=
 		\int_{\R^{dr}}
 		f^{(s+r)}
 		[V_{s+r }]
 		\d v_{s+1 } \cdots \d v_{ s+ r } . 
 		$
 		We subtract these two identities  and integrate against $\d b_K (\omega) $ to prove our claim. 
 		 	\end{proof}

\begin{definition}
\label{definition 6}
    For $N\geq M$, $1 \leq s \leq N$ and $ 0 \leq k \leq m $ we define the linear operator 
    $$
    \c_{s,  s + k }^N 
    : L^1_{\mathrm{sym}}  (\R^{d(s+k)}) \rightarrow L^1_{\mathrm{sym}}  (\R^{ds})
    $$ 
    according to the formula 
    	\begin{equation}\label{C^N formula}
	    \c^N_{s,s+ k } :=
	    \sum_{n=1}^{M-k} \beta_{ k +n} \frac{N}{\binom{N}{ k +n}} \binom{N-s}{ k }
	    \sum_{1\leq i_1 < \cdots <i_n \leq s} C_{i_1,\cdots,i_n}^{s, k +n,n}
	    \ . 
	\end{equation}
\end{definition}

\begin{remark}
\label{finite remark}
	It is straightforward to verify that, for each $N \in \N$, there is only finitely many  operators 
	that are non-zero. In particular, $\c_{s,  s + k }^N = 0 $ for any $s\in \N$ satisfying $s + k > N$. 
\end{remark}

	We are now ready to record the BBGKY hierarchy.

	\begin{lemma}\label{prop hierarchy}
	For all $N\geq M$ and $1 \leq s \leq N$.
	Let  $f_{N} $ denote the solution of the Master equation \eqref{master eq}, and let  $(f_N^{(s)})_{s\in\N}$ be its   its sequence of marginals, defined in \eqref{marginals}. Then, 
	it holds that  
		\begin{equation}
		\partial_t f_N^{(s)}
		= \sum_{  k   = 0  }^m
		\c_{s,s+k}^N 
		f_N^{(s+k)} \ . 
		\end{equation}
	\end{lemma}

	\begin{proof}
		First, following the same argument of the proof of Lemma \ref{prop identity}, we may verify that  for 
		$ s < i_1 < \cdots < i_K$ it holds that 
		\begin{equation}
		\tr_{s+1, \ldots, N}  \big(\Omega_{ i_1 \cdots i_K } f   \big) 
		= 0  \ . 
		\end{equation}
	Next, we use the following decomposition of the set of ordered indices
		\begin{equation}
		\sum_{ i_1 < \cdots  < i_K }
		\tr_{s+1, \ldots, N}  \big(\Omega_{ i_1 \cdots i_K } f   \big) 
		=
		\
		\sum_{n=1}^K
		\
		\sum_{    \substack{  i_1 < \cdots  <   i_K  \\   i_n \leq s < i_{n+1} }  } 
		\tr_{s+1, \ldots, N}  \big(\Omega_{ i_1 \cdots i_K } f   \big) 
		\ . 
		\end{equation}
		where for notational convenience we denote $i_{K +1} =N +1$. 
		In other words, $n$   counts the number of the indices $\{  i_\ell  \}$ that are less than or equal to $s $. 
		In addition, we note that 
		\begin{equation}\label{ineq 1}
		s < i_{n + 1 } < \cdots <  i_{ K } \leq N 
		\quad \implies \quad 
		N \geq s + (K - n) \ . 
		\end{equation}
		We implement Eq. \eqref{ineq 1} by means of a characteristic function $ \1_{K, n} \equiv  \mathds{1} ( N \geq s + K - n)$. 
		Thus, we may write thanks to Eq. \eqref{eq 1} and Lemma  \ref{prop identity}
		\begin{align}
		\partial_t f_N^{ ( s )  }   											\nonumber 
		 & \ =  \  \sum_{K  = 1 }^M  
		\beta_K  \frac{N}{ \binom{N}{K}} 
		\sum_{n = 1 }^K 
		\,  \1_{K, n} \! 
		\sum_{    \substack{  i_1 < \cdots  <   i_K  \\   i_n \leq s < i_{n+1} }  } 
		\tr_{s+1, \ldots, N}  \big(   \Omega_{   i_1 \cdots i_K   }    f_{N} \big)   
		\\
	  & 											  
		=  \sum_{K  = 1 }^M  
		\beta_K  \frac{N}{ \binom{N}{K}} 
		\sum_{n = 1 }^K 
		\,  \1_{K, n} \! 
		\sum_{    \substack{  i_1 < \cdots  <   i_K  \\   i_n \leq s < i_{n+1} }  } 
		C^{s,K,n}_{ i_{1} \cdots   i_{n }}  f_N^{(s + K - n )}  \ . 
		\end{align}
		Note that $C^{s,K,n}_{ i_{1} \cdots   i_{n }}  f_N^{(s + K - n )}   $  does not depend on the indices $ i_{ n +1 } < \cdots  < i_{n +r }$, so these can be summed out. We find that 
		\begin{align}
		\partial_t f_N^{ ( s )  }   
		& 											\nonumber  
		=  \sum_{K  = 1 }^M   \beta_K  \frac{N}{ \binom{N}{K}} 
		\sum_{n = 1 }^K 
		\,  \1_{K, n} \! 
		\sum_{    1 \leq   i_1 < \cdots  <   i_n \leq s  } 
		\bigg(       \sum_{ s + 1   \leq    i_{n + 1 } < \cdots i_K \leq N  }  \! \! 1      \bigg) 
		C^{s,K,n}_{ i_{1} \cdots   i_{n }}  f_N^{(s + K - n )}
		\\
		& = 				
		\sum_{K  = 1 }^M 
		\beta_K  \frac{N}{ \binom{N}{K}} 
		\sum_{n = 1 }^K 
		\,  \1_{K, n} \! 
		\sum_{   1 \leq   i_1 < \cdots < i_n \leq s   } 
		\binom{N - s}{  K  -  n }
		C^{s,K,n}_{ i_{1} \cdots   i_{n }} 
		f_N^{(s + K - n )}\, .
		\end{align}
			Further, note that
			$
			\1_{K, n} C^{s,K,n}_{ i_{1} \cdots   i_{n }}  = C^{s,K,n}_{ i_{1} \cdots   i_{n }}  $.
			Finally,  we  make the substitution $r = K - n$ to obtain
			\begin{align}
		\partial_t f_N^{ ( s )  }   
		& 											\nonumber  
		=  \sum_{r  = 0 }^{M-1} \sum_{n = 1 }^{M-r}  \beta_{r+n}  \frac{N}{ \binom{N}{r+n}} \binom{N-s}{r}
		\sum_{   1 \leq   i_1 < \cdots < i_n \leq s   }
		C^{s,r+n,n}_{ i_{1} \cdots   i_{n }}  f_N^{(s + r )}
		\\
		& = 				
        \sum_{r =0}^m \c^N_{s,s+r } f_N^{(s+r)}
		\end{align}
		where on the last line we   recall that $M-1 = m$.
		This finishes the proof.
	\end{proof}

	The next result shows that
	the operators that drive the BBGKY hierarchy
 fit  the abstract framework   introduced in Section \ref{section main results}.

	\begin{lemma}\label{lemma estimates}
Assume $M/N \leq \ve \in (0,1)$. 
Then, 	the    operators $  ( {\c}_{s,s+k}^N )_{k = 0   }^{ m  } $ satisfy Condition \ref{condition 1} with constants $(R_k)_{k=0}^m$
given by 
\begin{equation}
     R_k  
     =
     2
     \sum_{ \ell =  k + 1   }^{M}
     \frac{ \beta_{\ell } }{ (1 - \ve)^{ \ell   }}
     \binom{ \ell  }{ k } 
\end{equation}
	\end{lemma}
	 
In the upcoming proof, 
we will make 
use  of the following two inequalities
 \begin{equation} 
 \label{asym formula}
( 1 - k/n)^k
\, 
\frac{n^k}{k!}
\leq
\binom{n}{k}
\leq 
\frac{n^k}{k!} \ , \qquad \forall n \in \N , \ \forall k \leq n \ , 
 \end{equation}
which can be easily derived by noting,
\begin{equation}
    (1-k/n)^k n^k = (n-k)^k \leq n(n-1) \cdots (n-(k-1)) \leq n^k.
\end{equation}
	
	\begin{proof} 
We assume without loss of generality that $s \leq N$, for otherwise $\c_{s , s +   k  }^N = 0$ for any $ k \geq 0$. 
First, 	recalling that 
$
\| C^{s,K,n}_{i_1 \cdots i_n  }   \|
\leq 
2
$ 
we find that 
	\begin{align}
	    \| 
	    \c_{s,s+k}^N
	    \|
	     & = 
	     \big\|
	      \sum_{n=1}^{M-k} \beta_{ k +n} \frac{N}{\binom{N}{ k +n}} \binom{N-s}{ k }
	    \sum_{1\leq i_1 < \cdots <i_n \leq s} C_{i_1,\cdots,i_n}^{s, k +n,n}
	     \big\| \\
	     &   \, 
	     \leq 
	     2
	      \sum_{n=1}^{M-k} 
	      \beta_{ k +n} 
	      \frac{N}{\binom{N}{ k +n}}
	      \binom{N-s}{ k }    
	      \binom{s}{n},
	      \label{lemma 6.4 eq 1}
	\end{align}
	where we have used the fact that
	$
	 \sum_{1\leq i_1 < \cdots <i_n \leq s}  = \binom{s}{n}
	$. 
	Next, we use Eq. \eqref{asym formula}
	to estimate,
	\begin{align}
	\label{lemma 6.4 eq 2}
	     N  \binom{N - s}{k } \binom{s}{n}
	     \leq 
	     N  \binom{N}{k } \binom{s}{n}
	     \leq 
	     \frac{N^{k+1}}{k! } \frac{s^n}{  n! } \ . 
	\end{align}
	Similarly, for the denominator we find
	\begin{align}
	\label{lemma 6.4 eq 3}
	\binom{N}{n+k}
	\geq 
	\Big( 
	 1  - (n+k)/N
	\Big)^{n+k}
	\frac{N^{n+k}}{(n+k)!}
	\geq 
	( 1 - \ve )^{n  + k }
	\frac{N^{n+k}}{ (n+k)!}
	\end{align}
	where, for the second inequality, 
	we used the fact that 
	$ (n+k)/N \leq M /N  \leq  \ve$.
	We put together Eqs. \eqref{lemma 6.4 eq 1}, \eqref{lemma 6.4 eq 2} and \eqref{lemma 6.4 eq 3}
	to find that
	\begin{align}
	      \| 
	    \c_{s,s+k}^N
	    \|
	       \  \leq \ 
	       2
	      \sum_{n=1}^{M-k} 
	      \beta_{ k +n} 
	      (1 - \ve)^{-(n+k)}
	     \binom{n+k}{k}
	     \frac{s^n}{N^{n-1}}
	  \   \leq  \ 
	     R_k \, s 
	\end{align}
	    where, in the second inequality, we have used the upper bound
	    $s^n  N^{ - (n -1 )} \leq s,$
	    followed by a change of variables $\ell = n + k $. 
This finishes the proof.  
	\end{proof}
	%	(1) Recall that $\| C_i^{s,K,1} \|\leq 2$. Thus, for all $ k = 1 , \ldots,  m  $ one finds that 
%\begin{equation}
%\|    \widetilde{\C}_{s,s+k}^N 		\|  
%\  \leq  \ 
%\frac{\beta_K N }{ \binom{N}{K} }\binom{N-s}{K-1} 2s
%\  \leq \ 
%\frac{N}{\binom{N}{K}} \binom{N-1}{K-1 } 2s 
%\  = \ 
%2K s 
%\end{equation}
%where in the last equality we used the identity $\binom{N}{K} = \frac{N}{K} \binom{N-1 }{K-1}$. 
%
%\vspace{1mm}
%
%\noindent (2) Similarly, one finds for the remainder term that
%\begin{equation}
%\|  \calR_{s,s +r }^N \|  
%\leq
%\sum_{ K = r + 2 }^M 
%\beta_K 
%\frac{ N }{     \binom{N}{K}  }
%\binom{N -s }{  r  }
%2 s^{K -r} \ . 
%\end{equation}
%In view of $ 0 \leq r < K  \leq M $, we may estimate the above coefficients as
%\begin{equation}
%\frac{\binom{N -s }{r }}{ \binom{N }{ K } }
%\leq 
%\frac{\binom{N }{r }}{ \binom{N }{ K } }
%= 
%\frac{ (K! / r! )}{  (N-r ) \cdots ( N - K  + 1 )    }
%\leq 
%\frac{ (K! / r! ) }{
%	N^{K-r} (1 - M/N)^{K-r }
%}  
%\leq 
%\frac{ M!  }{
%	(1 - M/N)^{M}
%}    
%\cdot 
%\frac{1}{	N^{K-r} }
%\end{equation}
%Finally, we combine the last bound with the inequality $s^{K-r } \leq s^2 N^{K - r -2 }$ to conclude that  for $ N_* = N_*(M)$ large enough it holds that 
%\begin{equation}
%\|  \calR_{s,s +r }^N \|  
%\   \leq  \ 
%\frac{ 2 M!  }{ 
%	(1 - M/N)^{M} } 
%\frac{s^2}{ N }
%\   \leq  \ 
%\frac{4 M!  s^2}{N } \  , \qquad  \quad N \geq N_*(M) \ . 
%\end{equation}
%Note that   we have used the fact that $\sum_K \beta_K = 1$. This finishes the proof. 

	\vspace{1mm}
	
	\subsection{Convergence of operators}
For $s \in \N$ and $ 0 \leq k \leq M$, we introduce the operator
	$$  \c_{s,s+k}^\infty : 
	L^1_{ \textnormal{sym}  } ( \R^{d(s+k)})  
	\rightarrow 
	L^1_{ \textnormal{sym}  }  ( \R^{ds })$$ 
	given by   
	\begin{align}
	( 
	\c_{s,s+k}^\infty 
	f^{(s+k)} 
	) 
	(
	V_{s  } 
	)
	& := 
	\beta_K  K 
	\sum_{i=1}^s 
	\int_{\S_{K }\times \R^{d k } }  
	%\int_{B_N^{(s + k-1 )}}d v_{s + 1 } \ldots d v_{s+ k -1 }  \\ 
	\Big(
	f^{(s+k)} 	
	( 
	V^{*i}_{s +  k} 
	)
	-
	f^{(s+k)} 
	(V_{s+k})
	\Big)
	\nonumber 
	\d b_K(\omega)
	\d v_{s+1} \cdots v_{s+ k }
	\ 
	\end{align}		
for $ k \geq 1$,  and 	with the obvious modification for $k=0$. Here, $V_{s+k}^*$ is as in Definition \ref{definition 5}. 

\vspace{1mm}

Our following result establishes convergence of operators, 
which in turn allow us to apply  
Theorem \ref{thm convergence}. 
In order to state it, we introduce on   $X = \oplus_{s \in \N } L^1_\sym  (\R^{ds})$ the linear  operators
\begin{align}
\label{def C^N}
     (\c^N F)^{(s)} 
      : =
     \sum_{ k = 0 }^m 
       \c_{s,s+k}^N  f^{(s+k )}
     \ , 
     \qquad F = (f^{(s)})_{s \in \N }
\end{align}
where $\c^N_{s,s+k}$ was defined in Definition \ref{definition 6}, 
	and 
	\begin{align}
	\label{def C^inf}
     (\c^\infty  F)^{(s)}  : =
     \sum_{ k = 0 }^m 
       \c_{s,s+k}^\infty  f^{(s+k )}
     \ , 
     \qquad F  = (f^{(s)})_{s \in \N } \ . 
\end{align}
	
	\vspace{1mm}
	
	\begin{lemma}
	\label{lemma convergence}
		 Let $\c^N$ be as in \eqref{def C^N}, and
		 $\c^\infty $ be as in \eqref{def C^inf}, respectively.
		 Then, $\c^N$ converges to $\c^\infty $
		in the sense of Definition \ref{definition convergence}. 
	\end{lemma}
	
	\vspace{1mm}
	% \begin{equation}\label{dual identity}
	% \< 
	% \C_{s , s +  k   }^N
	% f^{(s + k   )} , \vp 
	% \>_{ B_N^{(s)}}
	% =
	% \frac{1 }{ k + 1  }
	% \frac{ N \binom{N-s}{ k }}{\binom{N}{k + 1 }}
	% \< 
	% f^{(s + k  )} , 
	% \mathcal{D}_{s+ k, s  } \vp_s 
	% \>_{ B_N^{(s +  k )}} \, .
	% \end{equation}
	
	\begin{proof}
	Let us fix $ k \in \{ 0 , \ldots, M-1 \}$.
First, we decompose the BBGKY operator into a \textit{leading order term}, and a 
\textit{remainder term}
\begin{equation}
\label{C decomposition}
\c_{s,s+k}^N = \Tilde{\c}_{s,s+k}^N 
+  
\calR_{s,s+k}^N.
\end{equation}
This decomposition follows from Eq. \eqref{C^N formula}--the leading order term  corresponds to the $ n  =1 $ contribution whereas the remainder term corresponds to the $ n \geq 2$ contribution.
Explicitly, 
we have
\begin{equation}\label{C tilde}
\widetilde{\c}_{s , s +   k  }^N 
: = 
\beta_{k +1}  \frac{N}{ \binom{N}{ k+1 }} 
\binom{N -s }{  k   }
\sum_{1 \leq i \leq s }
C_{i}^{s , k + 1  , 1 } 			 
\end{equation}
and 
\begin{equation}\label{reminder}
\calR_{s,s +k }^N :=
\sum_{n=2}^{M-k} \beta_{ k +n} \frac{N}{\binom{N}{ k +n}} \binom{N-s}{ k }
\sum_{1\leq i_1 < \cdots <i_n \leq s} C_{i_1,\cdots,i_n}^{s, k +n,n}
\end{equation}

\vspace{1mm}

	The following is enough to prove our claim. 	Let $F_{N} =(f_N^{(s)})_{s \in \N }  \in X$ converge weakly to $F = (f^{(s)})_{s \in \N } \in X$. 
		Then, for all $s \in \N $ and  $\vp_s \in L^\infty  (\R^{ds})$: 
		\begin{enumerate}
			\item  There holds 
			$$
			\lim_{ N  \rightarrow \infty }  
			\<   \widetilde{\c}_{s , s+ k }^N   f_N^{(s+k)}  , \vp_s  \> 
			=
			\<   \c_{s , s+ k }^\infty    f^{(s+k)}  , \vp_s  \>  	   \ . 
			$$
			\item There holds 
			$$
			\lim_{ N  \rightarrow \infty }  
			\<   \calR_{s , s+ k  }^N   f_N^{(s+r)}  , \vp_s  \>  
			=
			0 \ . 
			$$
		\end{enumerate}

	\vspace{1mm}

		We shall assume for simplicity that $k\geq 1$, the case $k=0$   being analogous.
		
		\vspace{1mm}
		
	 \textit{Proof of (1)}	Let us denote by $\mathcal{D}_{s+ k ,s } =  (  \c_{s,s+k}^\infty  )^* : L^\infty (\R^{ds }) \rightarrow L^\infty (\R^{d (s+k)})$ the Banach space adjoints of the limiting collisional operators. 
		In particular, they admit the representation 
		\begin{align}
		( 
		\mathcal{D}_{s+ k, s  }
		\vp_s ) 
		(
		V_{s + k } 
		)
		& 
		= 
		\beta_{k+1} (k+1)
		\sum_{i=1}^s
		\int_{\S_{K }}  
		%\int_{B_N^{(s + k-1 )}}d v_{s + 1 } \ldots d v_{s+ k -1 }  \\ 
		\Big(   
		(  \vp_s \otimes \1_k   )
		(  
		V_{s + k   }^{*i}
		)
		-
		( \vp_s \otimes  \1_k   )  ( V_{s + k   } ) 
		\Big)  						\nonumber 
		\d b_K(\omega )  \ . 
		\end{align}	
		where $\1_k $ is the $dk$-dimensional identity. A straightforward calculation based on a change of variables shows that 
		\begin{equation}\label{dual identity}
		\langle
		\widetilde{\c}_{s , s +  k   }^N
		f_N^{(s + k   )} , \vp 
		\rangle 
		=
		\frac{1 }{  k+1  }
		\frac{ N \binom{N-s}{  k   }}{\binom{N}{ k+1  }}
		\langle
		f_N^{(s + k  )} , 
		\mathcal{D}_{s+ k, s  } \vp_s 
		\rangle
		\, .
		\end{equation}
		Since  $\mathcal{D}_{s+ k, s  } \vp_s  \in L^\infty  (\R^{d (s + k)})$ we use   weak convergence of the marginals to calculate that 
		\begin{align}
		\lim_{N\rightarrow \infty }				\nonumber  
		\langle 
		\widetilde{\c}_{s , s +  k   }^N
		f_N^{(s + k   )} , \vp 
		\rangle 
		& =
		\lim_{ N  \rightarrow \infty } 		
		\Bigg( 
		\frac{1 }{  k+1   }
		\frac{ N \binom{N-s}{  k   }}{\binom{N}{  k+1 }}
		\Bigg) 
		\lim_{ N  \rightarrow \infty }
		\langle 
		f_N^{(s + k  )} , 
		\mathcal{D}_{s+ k, s  } \vp_s 
		\rangle  \\ 	
		& = 			\nonumber   
		\langle  f^{ (s+ k)}  ,  \mathcal{D}_{s+ k, s  } \vp_s 	\rangle  \\
		&  =
		\langle  \c_{s , s +  k   }^\infty  f^{ (s+ k)}  ,    \vp_s 	\rangle  \ . 
		\end{align}
		This finishes the proof of (1). 
		
		\vspace{1mm}
		
		\noindent  \textit{Proof of (2)} 
		First, we establish a norm estimate for the remainder term. 
		The same analysis done in Lemma \ref{lemma estimates} 
		can be carried out for the remainder term to   find that 
		for $N \geq \ve^{-1} M $
		\begin{equation}
		    \|
		    \calR_{s,s+k}^N  
		    \| 
		    \leq 
		    2
		    \sum_{n=2}^{M-k}
		    \frac{\beta_{n+k}}{(1 - \ve)^{n+k}}
		      \binom{ n+ k }{ k }
		      \frac{s^{n}}{N^{  n -1 }} \ . 
		\end{equation}
For $s \leq N$, 	and $ n \geq 2 $	we can now use the alternative upper bound  
		$s^n = s^2 s^{n -2 } \leq s^2 N^{n-2 } $ to find that 
		the following estimate holds
		\begin{equation}
		\label{remainder estimate}
		      \|
		    \calR_{s,s+k}^N  
		    \| 
		    \leq 
		    C_k \frac{s^2}{N}
		\end{equation}
		where $
		C_k = \sum_{\ell = k +2}^{M}
		      (1 - \ve)^{-\ell}
		      \beta_{\ell} 
		      \binom{ \ell}{ k } $. 
		Next, fix $s \in \N $ and note that--thanks to weak convergence and the Uniform Boundedness Principle--the quantity $K_s = \sup _{N \in \N }    \|  f_N^{(s)}  \|_{L^1_\sym  (\R^{ds})  }$ is finite. 
		Thus, we find that the following estimate  holds
		\begin{equation}
		|	 \<   \calR_{s,s+k}^N f_{s,s+k}^N , \vp_s  \>	| 
		\, \leq  \, 
		K_s  \, \|  \calR_{s ,s+k}^N   \| 
		\,  \| \vp_s  \|_{L^\infty }
		\, \leq \, 
		\frac{	K_s \,  C_k \, s^2}{N}
		 \|  \vp_s  \|_{L^\infty }
		\end{equation}
		from which our claim follows after taking the $N \rightarrow \infty $ limit. 
	\end{proof}

	\begin{proof}[Proof of Theorem \ref{thm bbgky to boltzmann}]
	    Lemma \ref{lemma estimates} implies that the operator
	    $\c^N$
	    satisfies Condition \ref{condition 1}. 
	    Similar arguments show that $\c^\infty$ satisfies Condition \ref{condition 2}. 
	    Further, Lemma \ref{lemma convergence} shows that $\c^N$
	    converges to $\c^\infty$ in the sense of Definition \ref{definition convergence}. 
	    In order to prove Theorem \ref{thm bbgky to boltzmann}, it suffices to apply Theorem \ref{thm convergence}
	    to any mild solutions of the BBGKY and Boltzmann hierarchies, respectively.
	\end{proof}

 	\begin{proof}[Proof of Theorem \ref{thm prop of chaos}]
Let $f_N $ be the solution of the Master equation \eqref{master eq},  $(f_N^{(s)})_{s \in \N }$   its sequence of marginals \eqref{marginals}, and $f_0 \in L^1 (\R^d)$    the initial datum  for which $f_{N}^{(s)} (0)$ converges pointwise 
weakly to $f_0^{\otimes s }$. 
	We  apply Theorem \ref{thm bbgky to boltzmann} to conclude that $(f_N^{(s)})_{s \in \N }$ converges in observables to
	$F = (f^{(s)})_{s \in \N }$--the solution of the Boltzmann hierarchy \eqref{boltzmann hierarchy} with initial data $F_0  = (f_0^{\otimes s })_{s \in \N}  $--over $[0,T]$. 
		
		\vspace{1mm}
		
		Finally, let $f (t,v)$ be the solution of the generalized Boltzmann equation \eqref{Boltzmann equation} with initial data $f_0$. 
		A straightforward calculation shows that $(f^{\otimes s})_{s \in \N }$ is a mild solution of the Boltzmann hierarchy \eqref{boltzmann hierarchy}.
		Because of Proposition \ref{corollary 2}, the Boltzmann hierarchy is well-posed. Uniqueness then implies that
		$f^{(s)} = f^{\otimes s}$ for all $s \in \N$. 
		Consequently, for all $\vp_s \in L^\infty (\R^{ds})$ it holds that 
		$$ \langle f_N^{(s)}(t ) ,  \vp_s \rangle  \longrightarrow \<  f(t)^{\otimes s } , \vp_s   \>  
		\quad \textnormal{ as } \quad 
		N \rightarrow\infty  $$
		uniformly in $t \in [0,T]$. 
		Since $T$ is independent of the initial conditions, and thanks to the global apiori bounds
		\begin{equation}
		\sup_{N \in \N} \sup_{s  \leq N }  \|  f_{N}^{(s)}  (t, \cdot)   \|_{L^1 } = 1 \, , \qquad \forall t  \geq 0 
		\end{equation}
		one may repeat the above argument to prove convergence for arbitrarily large $t \geq 0$. This finishes the proof. \end{proof}

	\section{Well-posedness} \label{appendix wp}
	
	\subsection{The  Hierarchies}
	In this subsection we address the question of well-posedness 
	of the finite and infinite hierarchy, respectively. 
	We only give a proof of Proposition \ref{corollary 2}, the other one being completely analogous.

	\begin{proof}[Proof of Proposition \ref{corollary 2}]
		For simplicity, let us denote $\bC \equiv \bC^\infty$. 
		Let $ \bF   =  ( \bf^{(s)} )_{s \in \N} \in \bX_{\bmu}  $ .
		Then, we obtain thanks to Condition \ref{condition 2}   the following estimate 
		\begin{align}
		\|								\nonumber 
		\Big( 
		\int_0^t 
		\bC    \big[   \bF    (\tau) \big]
		\,  \d \tau 
		\Big)^{(s)}  		\|_{X^{(s)}}	 
		& \leq
		\int_0^t 
		\|								\nonumber 
		\Big( 
		\bC    \big[   \bF    (\tau) \big]
		\Big)^{(s)}  		\|_{X^{(s)}}
		\,  \d \tau 		\\ 
		&  \leq 			\nonumber
		\int_0^t 
		\sum_{k=0}^m
		s
		\rho_k 
		\|	    \bf^{(s + k )}(\tau )		\|_{X^{(s)}}
		\, \d \tau \\
		&  =  				\nonumber
		\int_0^t 
		\sum_{k=0}^m
		s
		\rho_k 
		e^{ - \bmu (\tau) (s + k )}
		e^{ \bmu (\tau) (s + k ) }
		\|	    \bf^{(s + k )}(\tau )		\|_{X^{(s)}}
		\, \d \tau \\ 
		& \leq 		 
		\nonumber 
		\int_0^t 
		\sum_{k=0}^m
		s 
		\rho_k  e^{- \bmu (\tau) (s + k  )}
		\, \d \tau   \, 	\| \bF \|_{\bmu }			\\
		& \leq
		T  
		\sum_{k=0}^m
		\frac{s}{ s+ k }
		\rho_k   e^{  - \bmu(t) (s +  k) } \, 	\| \bF \|_{\bmu  }	\nonumber  \\
		& 
		\leq 
		\Big(
		T \sum_{k = 0}^m \rho_k  e^k
		\Big)
		e^{- \bmu(t)s		} 
		\| \bF \|_{\bmu} 
		=
		\theta
			e^{- \bmu(t)s		} 
		\| \bF \|_{\bmu} 
		\label{estimate wp}
		\end{align}
		where we have defined 
		$\theta := T \sum_{k = 0}^m \rho_k  e^k  \in (0,1) $. 
		
		\vspace{1mm}
		
		On $\bX_{\bmu }$ we introduce  the map  
		$\bF \mapsto F_0 
		+
		\int_0^t \     \bC   \bF   (\tau )    \d \tau  =: \M [ \bF ]$. 
		Linearity of $\bC$, and the  estimate contained in Eq. \eqref{estimate wp},  
		imply that 
		$\| \M[ \bF] - \M[ \bG] \|_{\bmu } 
		\leq
		\theta 
		\|  \bF - \bG \|_{\bmu}$ 
		for all $\bF,\bG \in \bX_{\bmu} $. 
		Therefore, $\M $ is a contraction.
		Let $r = (1 - \theta)^{-1} \theta \in (0,\infty)$,
		and define $ R  := r  \|  F_0\|_{X_0}$.
		Then, estimate \eqref{estimate wp} and the triangle inequality show that 
		$$ 
		\|  \M [\bF] - F_0 \|_{\bmu} 
		\leq 
		\theta 
		\|  \bF \|_{\bmu  } 
		\leq 
		\theta 
		\|  \bF - F_0  \|_{\bmu  } 
		+
		\theta 
		\|  F_0 \|_{\bmu  } 
		\leq R 
		$$
		whenever $ \|   \bF -  F_0 \|_{\bmu } \leq R $. Therefore, $\M$ maps the ball $B_R (F_0 ) \subset \bX_{\bmu}$ or radius $R$ around $F_0$, into itself. The conclusion of the theorem now  follows from Banach's fixed point theorem. The continuity estimate also follows easily from our considerations. 
	\end{proof}

	\subsection{The Boltzmann Equation}
	The main goal of this subsection is to prove Proposition \ref{prop global wp}. 
	First, we prove the following two lemmas. 
	
	\begin{lemma}[Local well-posedness]\label{lemma local wp}
		For all $f_0 \in L^1(\R^d)$   
		there exists $0 < T_* = T_* ( \|  f_0\|_{L^1 } )$ such that
		there is a unique mild solution  $f \in C([0,T_*] , L^1(\R^d ))$ to 
		the Boltzmann equation \eqref{Boltzmann equation}  with initial data $f_0$. 
	\end{lemma}

	\begin{proof}
		The operators $Q_K$ satisfy the following estimates: for $f,g \in L^1$ there holds 
		\begin{align}
		\|  Q_K (f)   \|_{L^1 }					\label{continuity estimate 1}
		& 	\leq 
		2   K 
		\|  f    \|_{L^1}^K  \ ,  \\
		\|   Q_K(f)  - Q_K(g)   \|_{L^1 }   
		& \leq 										\label{continuity estimate 2}
		2  K^2 
		\big(
		\|		f	\|^{K-1 }_{L^1}  + \|	g 	\|^{K -1 }_{L^1}
		\big) 
		\|  f - g  \|_{L^1}  \ . 
		\end{align}
		Thanks to  these estimates, a  proof based on a fixed-point argument  
		shows that there is a unique solution to the integral equation
		$ f(t) = f_0 + \int_0^t  \sum_{K=1}^M \beta_K Q_K [ f(s) ,\ldots, f(s)] \d s $ . We leave the details to the reader. 
	\end{proof}
	
	\begin{lemma}[Conservation of mass]
		Let $f  \in C( [0,T_*] , L^1(\R^d) )$ be the continuous solution of  the Boltzmann equation in mild form \eqref{Boltzmann equation}. Then, 
		\begin{equation}
		\int_{\R^d }
		f(t,v) \d v =
		\int_{  \R^{d} } 
		f_0 (v) \d v \ , 
		\qquad 
		\forall t \in [0,T] \ . 
		\end{equation}
	\end{lemma}
	\begin{proof}
		Thanks to the estimate  \eqref{continuity estimate 2}, it is easy to show that the map 
		$ t \mapsto  Q_K [ f(t)] \in L^1(\R^d )$  is continuous. 
		It then  follows   that   $  f \in C^1 \big(  (0,T_*)  , L^1(\R^d) \big)$  and Eq. \eqref{Boltzmann equation} holds in the strong sense. 
		Consequently, we may calculate thanks to a change of variables that 
		\begin{equation}
		\partial_t 
		\int_{ \R^d  }
		f(t,v) \d v
		=
		\int_{\R^d }
		\partial_t
		f(t,v)
		\d v 
		=
		\sum_{K=1}^M
		\int_{  \R^{d} }
		\beta_K Q_K [ f(t  ) ,\ldots, f( t ) ] (v) 
		\  \d v
		= 0 \   , \qquad \forall t \in (0,T_*) \ . 
		\end{equation}
		This finishes the proof. 
	\end{proof}

	\begin{proof}[Proof of Proposition \ref{prop global wp}]
		Let $f \in C^1 \big(  [0,T_*]  , L^1(\R^d )  \big)$
		be the solution to the Boltzmann equation \eqref{Boltzmann equation}, with initial data $f_0$ satisfying $\int_{  \R^{d} }f_0 (v ) \d v  = 1$ and $\|  f_0\|_{L^1} \leq 1$, whose existence is guaranteed by Lemma \ref{lemma local wp}. 
		Our goal will be to show that
		$ \| f(t) \|_{L^1} \leq 1$ for all $ t \in (0,T_*)$, after possibly reducing $T_*$ by a constant depending only on $M$ and $\{  \beta_K \}_{K=1}^M$. 
		One may then patch the solutions obtained by Lemma \ref{lemma local wp} to obtain global well-posedness. 
		
		\vspace{1mm}
		
		First, we note that thanks to   conservation of mass,
		the collisional operators given in \eqref{collision operators}, when acting on $f$,  may be written as
		\begin{equation}
		Q_K [  f(t) , \ldots, f(t)  ]
		=
		Q_K^{(+)} [f (t) , \ldots, f(t)]
		- 
		 K f(t) \ , \qquad K =1 , \ldots, M \ , 
		\end{equation}
		where  $Q_K^{(+)} : L^1 (\R^d)^K \rightarrow L^1 (\R^d )$ corresponds to the 
		\textit{gain term}
		\begin{equation}
		Q_K^{(+)} [f_1 , \ldots, f_K]
		(v_1)
		=
		 K 
		\int_{\S_K \times \R^{dK}}
		(   \otimes_{\ell=1}^K f_\ell )  (T^\omega_K V_{K})
		\ 
		\d b_K (\omega)
		\d v_{2}
		\cdots 
		\d v_K \ . 
		\end{equation}
		Consequently, $f$ satisfies the equation 
		\begin{equation}
		\textstyle 
		\partial_t f 
		+ 
		\alpha 
		f 
		=
		\sum_{K =1 }^M \beta_K
		Q_K^{(+)}
		[ f, \ldots, f  ]  
		\end{equation}
		where $\alpha = \sum_{K=1}^M \beta_K K > 0 $. 
		Thus, Duhamel's formula   implies that    
		\begin{equation}\label{integral gain term}
		f(t)
		=
		e^{- \alpha t} f_0
		+
		\int_0^t
		e^{- \alpha (t -s )}
		\sum_{K =1 }^M \beta_K
		Q_K^{(+)}
		[ f(s), \ldots, f(s) ]  
		\d s \ , \qquad t \in (0,T_*) \ . 
		\end{equation}
		Next,  we  adapt the main ideas of the authors in \cite{gamba}\footnote{We note that the authors consider Picard iterates for an 
		equation similar to ours, but in Fourier space and in different functional spaces.} 
		and give only a sketch of the proofs. Indeed, we consider the sequence of Picard iterates $ \{  f_n \}_{n \geq 0 }$ defined as 
		\begin{align}
		\textstyle 
		f_0 (t)  & := f_0  \\ 	
		\textstyle 
		f_{n+1} (t) &  := e^{- \alpha t} f_0
		+
		\int_0^t
		e^{- \alpha (t -s )}
		\sum_{K =1 }^M \beta_K
		Q_K^{(+)}
		[ f_n (s), \ldots, f_n (s) ]  
		\d s \ \qquad  n \geq 1 \ .
		\end{align}
	We will use the iterates to show that $ \|  f(t) \|_{L^1 } \leq 1$. Indeed, if $ \| f_0 \|_{L^1} \leq 1$, an induction argument shows that $ \| f_n(t) \|_{L^1 }   \leq 1$ for all $ n \geq 0 $ and $ t \in (0,T_*)$. 
		Next, note that the gain operators $Q_K^{(+)}$ satisfy the estimate \eqref{continuity estimate 2}
		with a possibly different constant. 
		Consequently, for a possibly smaller  $T_*$, the following contraction estimate is satisfied thanks to 
		\eqref{continuity estimate 2}
		\begin{equation}
		\sup_{t\in [0,T_*]}
		\| f_{n+1 } (t) - f_n  (t)	\|_{L^1} 
		\leq   \lambda 
		\sup_{t\in [0,T_*]}
		\| f_{n } (t) - f_{n-1}  (t)	\|_{L^1}  \ , \qquad n \geq 1 
		\end{equation}
		for some fixed $\lambda \in (0,1)$. 
		Thus, the sequence $f_n $ converges to the (unique) solution of Eq. \eqref{integral gain term}. 
		We conclude  that    $\|  f(t) \|_{L^1} = \lim_{n\rightarrow \infty} \| f_n(t) \|_{L^1} \leq 1 $ for all 
		$t \in (0,T_* )$. This finishes the proof. 
	\end{proof}

\appendix 
	\section{Markov Processes}\label{appendix markov}
	
	\subsection{Review of the general theory}
	We give  a brief review of the   basic notions and results from the theory of  Markov processes that we use to construct our   model; we follow closely the discussion in \cite[Chapter 4]{EthierKurtz}.
	In what follows, we let  $(\Sigma, \mathscr{F}, \mathbb{P})$ be a probability space and
	$ E$ a locally     compact metric space, with its Borel sets $\mathscr{B}(E)$. 
	
	\vspace{1mm}
	
	\noindent \textit{Continuous time.} Let us define what we understand for a (continuous-time) Markov process. 
	
	\begin{definition}
		A stochastic process
		$
		\bX  =    (    X(t)   ) _{   t = 0 }^\infty  :  \Sigma \times [ 0, \infty ) \rightarrow E 
		$
		is called a \textit{Markov process} if 
		\begin{equation}
		\mathbb{P}
		\Big(
		X( t + s ) \in B 
		\,   \big| \, 
		\mathscr{F}_t^X 
		\Big) 
		=
		\mathbb{P} 
		\Big(
		X(t + s ) \in B 
		\,   \big| \, 
		\sigma \big(X( t ) \big)
		\Big) 
		\qquad 
		\forall  t ,s \geq 0 , \ \forall  B \in \mathscr{B}(E) 
		\end{equation}
		where $\mathscr{F}_t^{\bX }  =      \sigma \big(X(  s )   \, : \,   0 \leq s \leq t   \big)$.
	\end{definition}
	
	Some Markov processes are characterized  by more  tractable objects. 
	Indeed,  let   $  (  T(t)   )_{t \geq 0  }$ be a semigroup on   $ C_b (E)$, the bounded real-valued continuous functions on $E$.
	
	\begin{definition}
		We say that the Markov process $\bX $ \textit{corresponds} to $ (  T(t) )_{t \geq 0  }$  if 
		\begin{equation}\label{transition}
		\mathbb{E}      \Big[   \vp \big( X(t +s )   \big)   \big|        \mathscr{F}_t^X   \Big]
		=
		\big( T(s) \vp  \big)  \big(   X(t)    \big)  \qquad \forall t ,s \geq 0  , \ \forall \vp \in C_b (E) \, . 
		\end{equation}
		
	\end{definition}
	
	If a Markov proccess corresponds to a semigroup, it is completely determined by it in the following sense. 
	
	\begin{proposition} \cite[Chapter 4  Proposition 1.6]{EthierKurtz}  
		Let $\bX$ be a Markov process that corresponds to $ (  T(t)  )_{t \geq 0  }$. Then, the finite dimensional distributions of $\bX $ are complety determined by 
		$ (   T(t) )_{t \geq 0  }$   and the   law of $X(0)$. 
	\end{proposition}

	\vspace{1mm}
	
	\noindent \textit{Discrete time}. 
	Let us define what we understand as a (discrete-time) Markov chain.

	\begin{definition}
		A discrete-time stochastic process
		$ \textbf{Y } = (  Y(k ) )_{k \in \N_0} : \Sigma \times \N_0 \rightarrow E  $	                is called a  \textit{Markov chain }
		if
		\begin{equation}
		\mathbb{P}
		\Big(
		Y(  n + k  ) \in B 
		\,   \big| \, 
		\mathscr{F}_n^Y
		\Big) 
		=
		\mathbb{P} 
		\Big(
		Y   (  n+ k  ) \in A 
		\,   \big| \, 
		\sigma \big(    Y   (n ) \big)
		\Big) 
		\qquad 
		\forall   n , k \in \N_0  , \ \forall  B \in \mathscr{B}(E) \, .
		\end{equation}
		where
		$\mathscr{F}_n^{ \textbf{Y } } = \sigma \big(  Y( k )  : k \in \{ 0 , \ldots,  n  \}   \big)  $.
	\end{definition}
	
	Similarly as before, we can specify Markov chains in terms of more concrete objects. To this end, we define transition functions. 
	
	\begin{definition}
		
		A function $ \mu : E \times \mathscr{B}(E) \rightarrow [ 0 , \infty  )$  is called a 
		\textit{transition function}
		if
		\begin{align}
		\begin{cases}
		& \mu(x , \cdot ) \in \textnormal{Prob} (E) \, , \qquad \forall x \in E   \\
		& \mu ( \cdot, B ) \in L^\infty (E) \, , \qquad \forall B \in \mathscr{B} (E)  
		\end{cases} \, .
		\end{align}
	\end{definition}
	
	\begin{definition}
		We say that the Markov chain $\textbf{Y}$	     has $\mu $ as a transition function, if 
		\begin{equation}
		\mathbb{P}
		\Big(
		Y(  n + k  ) \in B 
		\,   \big| \, 
		\mathscr{F}_n^Y
		\Big) 
		=
		\mu \big(   Y(n) , B     \big) \, , \qquad n\in \N_0 , \ B \in \mathscr{B}(E) \, .
		\end{equation}
	\end{definition}
	
	Heuristically, transition functions correspond to  the probabilities for the Markov chain to go from one state to the next one. That one  may always construct Markov chains with prescribed transition functions and initial laws is the content of the following result.

	\begin{proposition}\label{PropEK}\cite[Chapter 4 Theorem 1.1]{EthierKurtz}
		For every transition function $\mu$ and probability measure $\nu \in \textnormal{Prob}(E)$, there exists a Markov chain $ \textbf{Y} $  that has $\mu$ as a transition function and   $\nu$ as  the law of $Y(0)$. 
	\end{proposition}

	\vspace{1mm}
	
	\noindent \textit{Jump processes.}
	If one is given a transition function $\mu $, one may construct Markov processes with  explicit transition semigroups; these are called jump processes, which we describe below.

	\begin{definition}
		Let $\textbf{Y}$	                be a Markov chain with transition function $\mu$.
		We define its \textit{generator} to be the linear map 
		$P : C(E) \rightarrow C(E)$ given by 
		\begin{equation}
		(P \vp)(x) := \int_{E} \vp(y )  \mu (x ,  d  y ) \qquad x \in E , \ \vp \in C_b(E) \, . 
		\end{equation}	
	\end{definition}

	Let $  (   M (t  ) )_{t  =  0 }^\infty  $	 denote  a Poisson process with parameter $\lambda$, independent of $ \textbf{Y}$. We let the \textit{jump process} associated to $   \textbf{Y} $ with parameter $ \lambda$              as the stochastic process 
	$ \textbf{V} =  (   V(t)  )_{ t   =  0 }^\infty  $ defined  by  
	\begin{equation}\label{jump}
	V(t)  : =  Y   \big(  M(t)    \big)   , \qquad t \geq 0 \, . 
	\end{equation}
	
	\begin{proposition}\cite[Chapter 4 Section 2]{EthierKurtz}
		The stochastic process	$  \textbf{V} $ 
		defined by \eqref{jump}
		is a Markov process that corresponds to the semigroup 
		$  \{     \exp\big(      t \, \lambda    (   P - \textnormal{Id}         )      \big)  \}_{t \geq 0 }$. 
	\end{proposition}

	Here we will give a sketch of the proof of the above proposition. 
	\begin{proof}
		Let $\phi \in C(E)$. The transition semigroup $T(t)$ for the Markov process $V(t)$ is defined as $T(s) \phi(V(t)) = \mathbb{E}[\phi(V(t+s)|\mathscr{F}_t]$. Using the memoryless property of the Poisson process $M(t)$ along with the Law of Total Probability we can calculate,

		\begin{align}
		T(s)\phi(V(t))  \nonumber 
		&= 
		\mathbb{E}[\phi(V(t+s))|\mathscr{F}_t] = \mathbb{E}[\phi(Y((M(t+s))))|\mathscr{F}_t] = \mathbb{E}[\phi(Y(M(t+s)-M(t)+M(t)))|\mathscr{F}_t] \\
		&= \sum_{k\geq 0} \mathbb{P}(M(t+s)-M(t) = k)       \nonumber  \mathbb{E}[\phi(Y(k+M_t))|\mathscr{F}_t] =\sum_{k\geq 0} e^{-\lambda s} \frac{(\lambda s)^k}{k!} P^k \phi(V(t)) \\
		&= \exp{(t\lambda(P - \text{Id}))} \phi(V(t))        
		\end{align}
	\end{proof}           
	
	\begin{definition} \label{poisson process}
	    A \textit{Poisson process} $M(t)$ with rate $\lambda >0$ is a stochastic process taking values on $\mathbb{N}$ with the conditions: 
	    \begin{enumerate}
	        \item M(0) = 0.
	        \item For all $s_i < t_i$ the increments $M(t_i) - M(s_i)$ are independent random variables.
	        \item $\mathbb{E}[M(t)] = \lambda t$.
	    \end{enumerate}
	    Furthermore, the Poisson process is a Markov process and thus has the ``memoryless" property, implying that its increments satisfy, \begin{equation}
	        M(t_i) - M(s_i) = M(t_i - s_i) \quad \forall s_i < t_i.
	    \end{equation}
	    
	\end{definition}

	%\bibliographystyle{plain}
	%\bibliography{biblio}

\end{document}